\documentclass[11pt]{article}
\usepackage[T1]{fontenc}
\usepackage[utf8]{inputenc}
\usepackage{amssymb,graphicx,mathtools,amsfonts}
\usepackage{amsmath}
\usepackage{amsthm}
\usepackage{xcolor}
\usepackage[colorlinks]{hyperref}
\hypersetup{
	colorlinks=true,
	linkcolor=blue,      
	citecolor=red
	}
\usepackage{dsfont}
\usepackage{a4wide}
\usepackage{esint}
\catcode`\@=11
\def\theequation{\@arabic{\c@section}.\@arabic{\c@equation}}
\catcode`\@=12
\graphicspath{{images/}}

\newtheorem{theorem}{Theorem}[section] 
\newtheorem{lemma}{Lemma}[section] 
\newtheorem{proposition}{Proposition}[section]

\theoremstyle{definition}
\newtheorem{definition}{Definition}[section]
\newtheorem{remark}{Remark}[section]
\numberwithin{equation}{section}
\allowdisplaybreaks

\newcommand{\R}{\mathbb{R}}
\newcommand{\N}{\mathbb{N}}
\newcommand{\X}{\mathcal{X}}

\newcommand{\sm} {\setminus}
\newcommand{\cts} {\hookrightarrow}

\newcommand{\ve}{\varepsilon}

\newcommand{\al} {\alpha}

\newcommand{\ga} {\gamma}

\newcommand{\om} {\Omega}
\newcommand{\De} {\Delta}
\newcommand{\la} {\lambda}

\newcommand{\noi} {\noindent}
\newcommand{\na} {\nabla}

\newcommand{\ra} {\rightarrow}
\newcommand{\real}{\mathbb{R}}

\newcommand{\rnn}{\mathbb{R}^{N}}
\newcommand{\lv}{\lVert}
\newcommand{\rv}{\rVert}
\newcommand{\weak}{\rightharpoonup}

\newcommand{\lb}{\left(}
\newcommand{\rb}{\right)}

\newcommand{\grad}{\nabla}

\newcommand{\rnnn}{\int_{\mathbb{R}^N}}

\newcommand{\no}{\nonumber}

\title{Poho\v{z}aev identity and the existence of normalized ground state solutions for variable exponent problems} 

\author{ }
 \author{Nidhi Nidhi\thanks{Department of Mathematics, Indian Institute of Technology Delhi, Hauz Khas New Delhi 110016,  India, \texttt{maz218520@maths.iitd.ac.in}}, Ambesh Kumar Pandey\thanks{Department of Mathematics, Indian Institute of Technology Delhi, Hauz Khas New Delhi 110016,  India, \texttt{pandey.ambesh190@gmail.com}} and K. Sreenadh\thanks{Department of Mathematics, Indian Institute of Technology Delhi, Hauz Khas New Delhi 110016,  India, \texttt{sreenadh@maths.iitd.ac.in}}}
 
\date{}

\begin{document}
\maketitle

\begin{abstract}
\noi In this article, we investigate normalized solutions for nonlinear problems involving variable exponents. To the best of our knowledge, normalized solutions have not been previously studied in this setting, and our results appear to be new. A key difficulty is that the standard scaling argument, which is important in the classical normalized solution approach, is no longer available in the variable exponent setup. To address this, we work with a constrained variational framework and establish the existence of a ground state solution. We further show that these solutions are $C^{1,\al}_{loc}(\rnn)$. Finally, we derive a Poho\v zaev-type identity adapted to the variable exponent structure in $\rnn,$ which is used to prove that the solution is a ground state.

\medskip
\noi\textbf{Keywords:} Poho\v{z}aev identity; Normalized solutions; Variational methods; Variable exponents; Regularity.

\medskip
\noi\textbf{Mathematics Subject Classification:} 35J60, 35J20, 35J92.
\end{abstract}

\section{Introduction}
This work is concerned with the study of the following variable exponent problem:
\begin{equation}\label{problem}\tag{$\mathcal{P}_{\lambda}$}
   \left\{
	\begin{aligned}
		-\Delta_{p(x)}u+|x|^k|u|^{p(x)-2}u&=\lambda |u|^{p(x)-2}u +|u|^{q(x)-2}u \quad \text{ in } \mathbb{R}^N,\\
		\int_{\mathbb{R}^N}\frac{|u(x)|^{p(x)}}{p(x)}{\rm {\rm dx}}&=c,
	\end{aligned}
    \right.
\end{equation}
  where $\la$ is a real parameter and $\De_{p(x)}u:=\operatorname{div}(|\na u|^{p(x)-2}\na u)$ is the non-homogeneous operator called $p(x)$-Laplacian, which reduces to standard $p$-Laplacian when $p(x)\equiv p$.
            
   To study problem \eqref{problem}, we assume that the exponents satisfy the conditions listed below.
   	\begin{itemize}
\item[$(p_H)$] The function $p:\rnn\to \R$ is Lipschitz continuous and $$1<\inf_{x\in \rnn}p(x)\le \sup_{x\in \rnn}p(x)<N.$$
\item[$(q_H)$] $k>0,\, \la\in \R$ and $p(x)<q(x)\ll p^*(x):=\frac{Np(x)}{N-p(x)}\,\, \forall x\in \rnn$. The notation $q \ll p^{*}$ means that
\[
\inf_{x \in \mathbb{R}^N} \big( p^{*}(x) - q(x) \big) > 0 .
\]
\end{itemize}
The $p(x)$-Laplacian naturally arises in the modeling of nonlinear electrorheological fluids and in various applications, including image processing, elasticity, and flow through porous media; we refer to \cite{Image_restore,ruzicka2007electrorheological} for an overview. The function spaces naturally linked to the $p(x)$-Laplacian are the variable exponent Lebesgue–Sobolev spaces, which have been studied in \cite{Variable_book,fan_lp_wp_spaces} and the references therein. The literature on variable exponent problems is extensive, so without attempting to be exhaustive, we refer to \cite{Alves_potential,Fan_potential,Ky_Ho_weig}
for works involving potentials $V(x)$. However, none of these studies
addresses a normalized solution, which is the focus of the present paper. 

\noindent
The study of normalized solutions has significant importance in physics. One of the motivations for studying such solutions is that it provides the stationary state (or standing wave solutions) to the nonlinear Schr$\ddot{\text{o}}$dinger (NLS) equations. A standing wave solution for a nonlinear Schr$\ddot{\text{o}}$dinger (NLS) equation given as follows:
	\begin{equation*}
		i\frac{\partial \psi}{\partial t}=-\Delta \psi -\mu |\psi|^{p-2}\psi,
	\end{equation*}
	is of the form $\psi(x,t)=e^{-i\lambda t}u(x)$, where $\lambda\in \mathbb{R}$ and $u\in H^1(\mathbb{R}^N)$ solves
	\begin{equation}\label{1.2}
		-\Delta u =\lambda u +\mu |u|^{p-2}u  \text{ in }\mathbb{R}^N.
	\end{equation}
	While addressing solutions to \eqref{1.2}, there exist two schools of thought. The initial approach involves fixing a $\lambda \in \mathbb{R}$ and thereafter looking for the critical points of the associated functional. This method has already been extensively employed; see, for instance \cite{biagi2025brezis}. The other method is to fix the $L^2$-norm and look for the solution to the following constrained problem
		\begin{equation*}
		\left\{
		\begin{array}{rcl}
			-\Delta	u+(-\Delta)^su & = & \lambda u +\mu|u|^{p-2}u \text{ in } \mathbb{R}^N,\\
			\left\| u \right\|_2^2 & = & \tau^2,
		\end{array}
		\right.
	\end{equation*}
	called the {\it normalized solution}. Recently, the study of normalized solutions has attracted the attention of many researchers. It started with the pioneering work of Jeanjean \cite{jeanjean1997existence}, where he obtained the existence of radial solutions for the following problem
    \begin{equation}\label{Norm_sol}
		\left\{ \begin{array}{rl}   	
			& 	-\Delta u  =  \lambda u +g(u)\;\;\text{in } \mathbb{R}^N,\\
			&  	\left\| u \right\|_2^2 =  c,
		\end{array}
		\right.
	\end{equation} 
    under some assumptions on $g$. In \cite{bartsch2012normalized}, the existence of infinitely many solutions to \eqref{Norm_sol} under the same assumptions was established.
	Further in \cite{noris2015existence}, the normalized solutions are discussed and described in the case of  bounded domains with Dirichlet boundary conditions;
	moreover, the problem in general bounded domains has been dealt by the authors in \cite{pierotti2017normalized}. 
   There has been extensive research on the existence of normalized solutions to nonlinear Schrödinger systems; interested readers can refer to \cite{gou2018multiple,bartsch2019multiple,noris2019normalized} for more information. More recently, regularity, multiplicity, existence, and non-existence results for normalized solutions involving mixed local and nonlocal operators have been investigated in \cite{giacomoni2025normalized,giacomoni2024normalized,nidhi2025normalized,nidhi2025existence}. In addition to NLS equations, the study of normalized solutions is also relevant in the study of quadratic ergodic mean field games systems, as highlighted in \cite{pellacci2021normalized}.

    In the constant exponent case, when $p(x)\equiv 2$, $N=2$ and $q=4$, equation \eqref{problem}
    reduces to the time-independent modified Gross-Pitaevskii equation, introduced independently by Gross \cite{Gross} and Pitaevskii \cite{Pitaevskii} in the study of Bose-Einstein condensation. The existence of $L^{2}$-normalized ground states for this equation has been established in \cite{L2_solution}. More recently, Wang and Sun \cite{Normalised_potential} extended these results to the case $1<p<N$ and proved the existence of $L^{r}$-normalized ground state solutions for $r=p$ and $r=2$. In the particular case $r=2$, their work yields at least two solutions with positive energy, one being a ground state and the other a higher-energy solution. Nguyen and R\v adulescu \cite{Nonlocal_potential} have also studied the nonlocal counterpart and obtained multiple solutions by employing the Lusternik-Schnirelmann category method.
    

 Motivated by the above discussion, our aim is to study normalized solutions in the variable exponent setup, a problem that has been largely unexplored to date. One of the main difficulties is that the classical method used for normalized solutions in the constant exponent case does not apply here. In problems with constant exponents, one often studies the functional
	\[
	E(u)=\frac{1}{p}\int |\nabla u|^p\,{\rm {\rm dx}} - \frac{1}{q}\int |u|^q\,{\rm {\rm dx}},
	\quad S(c)=\{u:\|u\|_{L^r}^r=c\},
	\]
	through the scaling
	\[ 
	(s\star u)(x):=e^{Ns/r}u(e^{s}x).
	\]
	This defines a one-parameter curve $s\mapsto s\star u$ along which one analyzes $E(s\star u)$ on the Poho\v zaev manifold. However, when $p=p(x)$ or $q=q(x)$, this scaling is not possible. Hence, we abandon the fiber-map construction and instead use constrained variational methods. We can observe that solutions of the problem \eqref{problem} can be obtained by looking for critical points of the functional
\begin{equation*}
	E(u)=\int_{\mathbb{R}^N}\frac{|\nabla u|^{p(x)}}{p(x)}~{\rm {\rm dx}}+\int_{\mathbb{R}^N}\frac{|x|^k|u|^{p(x)}}{p(x)}~{\rm {\rm dx}}-\int_{\R^N}\frac{|u|^{q(x)}}{q(x)}~{\rm {\rm dx}}
\end{equation*}
on the constraint
\begin{equation*}
	S(c):=\left\{u\in\X \text{ such that }\int_{\R^N}\frac{|u|^{p(x)}}{p(x)}\, {\rm {\rm dx}}=c\right\}.
\end{equation*}
Inspired by \cite{Jeanjean_potential,Normalised_potential}, we consider the following localized version of the minimization problem above
\begin{equation}\label{constrained_min}
    \gamma_c=\inf_{S(c)\cap B_\sigma} E(u).
\end{equation}
Here for given $\sigma>0$, we define $B_\sigma:= \big\{u\in \X|\lv u\rv_\X\le \sigma\big\}.$ Our first aim is to prove that for every $\sigma>0$ there exists 
$c_{0}=c_{0}(\sigma)>0$ such that $S_{c}\cap B_{\sigma}\neq\emptyset$ whenever $c<c_{0}$. We then show that every minimizing sequence for $\gamma_{c}$ is compact for $c<c_{0}$, 
and therefore $\gamma_{c}$ is achieved by some $u_{c}\in S_{c}\cap B_{\sigma}$. 
To ensure that the minimizers of \eqref{constrained_min} are critical points of $E$ on $S_{c}$, 
we need to show that the minimizers do not lie on the boundary of $B_{\sigma}\cap S_{c}$. 
Once this is established, then it is classical that any minimizer $u$ admits a Lagrange multiplier 
$\lambda=\lambda(u)\in\mathbb{R}$ such that
\[
-\Delta_{p(x)}u + |x|^{k}|u|^{p(x)-2}u - |u|^{q(x)-2}u 
= \lambda |u|^{p(x)-2}u.
\]

The content of the paper is organized as follows. Section \ref{prelim} 
introduces the functional setting and the key lemmas used throughout 
the paper. In Section \ref{reg}, we establish regularity of weak normalized 
solutions, proving that they belong to $C^{1,\alpha}_{loc}(\mathbb{R}^N)$. 
Section \ref{pohoz} contains the proof of a Poho\v zaev-type identity for the 
variable exponent case in $\mathbb{R}^N$, which plays an important role in 
showing that the weak solutions obtained are ground states. Finally, in 
Section \ref{existence}, we establish the existence of a normalized ground state 
weak solution.

\section{Preliminaries and Statement of Main Results}\label{prelim}
In this section, we recall several definitions and preliminary results concerning variable exponents and the corresponding function spaces. We subsequently state the main result and establish the functional framework necessary for the analysis of \eqref{problem}.

We begin with the following notation: for any function $\Phi(x) \in C(\overline{\Omega}, \mathbb{R}) $, let
$$
\Phi^+ := \sup_{x \in \R^N} \Phi(x), \qquad \Phi^- := \inf_{x \in \R^N} \Phi(x).
$$

We define the function space
$$\mathcal{C}_+(\R^N):=\{\Phi\in C(\R^N):1<\inf_{x\in\R^N}\Phi(x)\le\sup_{x\in \R^N}\Phi(x)<\infty\}.$$

Let $\om$ (bounded or unbounded) be an open subset of $\R^N$. For $p\in \mathcal{C}_+(\overline{\Omega})$, define 
\begin{equation*}
    L^{p(x)}(|x|^k,\om) = \left\{ u : \om \to \real : u  \text{ is measurable and } \int_\om |x|^k|u(x)|^{p(x)}{\rm {\rm dx}} < \infty \right\}.
\end{equation*}
This is a Banach space endowed with the norm, known as the Luxemburg norm, given by
\begin{equation*}
    \lv u \rv_{L^{p(x)}(|x|^k,\om)} = \inf \left\{ \eta > 0  : \int_{\om} |x|^k\left| \frac{u(x)}{\eta} \right|^{p(x)} {\rm {\rm dx}} \le 1 \right\}.
\end{equation*}
The following H\"{o}lder-type inequality analogous to that of the classical $L^p{(\om)}$ spaces holds for the variable exponent Lebesgue spaces.
\begin{lemma}[\cite{Variable_book}]
	Let $p\in \mathcal{C}_+(\overline{\Omega})$ such that  $ \displaystyle \frac{1}{p(x)} + \frac {1}{p'(x)}=1$.
	Then for any  $u \in L^{p(x)}(\om)$ and $ v\in L^{p'(x)}(\Omega)$ we have
	$$\left| \int_{\Omega} uv {\rm {\rm dx}} \right|\leq \Big(\frac{1}{p^{-}} + \frac{1}{p^{'-}} \Big) \lv u \rv_{L^{p(x)}(\om)}
	\lv v \rv_{L^{p'(x)}(\Omega)}. $$
\end{lemma}
\noindent  To handle the Luxemburg norm, we define the modular function $\rho:L^{p(x)}(|x|^k,\Omega)\ra \real$ given as
$$\rho(u) = \int_\Omega |x|^k|u|^{p(x)}{\rm {\rm dx}}.$$
The relations between Luxemburg norm
$\lv \cdot \rv_{L^{p(x)}(|x|^k,\Omega)}$ and 
the corresponding modular function $\rho(\cdot)$ are given as follows:
\begin{lemma}[\cite{KIM2010,Ky_Ho_weig}]\label{modular_ineq}
	Let $u \in L^{p(x)}(|x|^k,\om)$, then 
	\begin{enumerate}
		\item[$(i)$]$ \lv u \rv_{L^{p(x)}(|x|^k,\om)}<1(=1;>1) \text{ if and only if } \rho(u)<1(=1;>1);$
		\item[$(ii)$] If $\lv u \rv_{L^{p(x)}(|x|^k,\om)}>1$, then $\lv u \rv_{L^{p(x)}(|x|^k,\om)} ^{p^{-}}\leq\rho(u)\leq\lv u \rv_{L^{p(x)}(|x|^k,\om)}^{p^{+}}$ ;
		\item[$(iii)$] If $\lv u \rv_{L^{p(x)}(|x|^k,\om)}<1$, then $\lv u \rv_{L^{p(x)}(|x|^k,\om)}^{p^{+}}\leq\rho(u)\leq\lv u \rv_{L^{p(x)}(|x|^k,\om)}^{p^{-}}.$
	\end{enumerate}
\end{lemma}
\begin{lemma}[\cite{Inbo_Sim_weigh,Ky_Ho_weig}]\label{lem:modular_cgs_lp}
	Let $u,u_{m}  \in L^{{p(x)}}(|x|^k,\om),~m=1,2,3,\cdots $. Then the following statements are equivalent
	\begin{enumerate}
		\item[$(i)$] $\displaystyle{\lim_{m\ra \infty} }\lv u_{m} - u \rv_{L^{p(x)}(|x|^k,\om)} =0;$
		\item[$(ii)$] $\displaystyle{\lim_{m\ra \infty}} \rho(u_{m} -u)=0;$
		\item[$(iii)$] $ u_{m} \text{ converges to $u$ in }  \Omega  \text { in measure and } \displaystyle{\lim_{m\ra \infty}} \rho(u_{m})= \rho(u).$
	\end{enumerate}
\end{lemma}
 

The variable Sobolev space $W^{1,p(x)}(\om)$ is defined as
\begin{equation*}
    W^{1,p(x)}(\om) = \left\{ u \in L^{p(x)}(\om) : |\grad u|^{p(x)} \in L^{p(x)}(\om) \right\},
\end{equation*}
Now, we define the following weighted variable Sobolev space
\begin{equation*}
    W^{1,p(x)}(|x|^k,\om) = \left\{ u \in W^{1,p(x)}(\om) : \int_\om |x|^k|u|^{p(x)}\,{\rm {\rm dx}} <\infty\right\},
\end{equation*}
equipped with the following norm
$$\lv u \rv_{1,p(x)} = \lv u \rv_{L^{p(x)}(|x|^k,\om)} + \lv \grad u \rv_{L^{p(x)}(\om)}.$$ 
We define the modular function $\rho_\X:W^{1,p(x)}(|x|^k,\om)\ra \real$ given as
$$\rho_\X(u) = \int_\Omega |\nabla u|^{p(x)}\,{\rm {\rm dx}}+\int_\Omega |x|^k|u|^{p(x)}\,{\rm {\rm dx}}.$$
We would like to point out that we can also define the following modular
$$\tilde\rho_\X(u) = \int_\Omega \frac{|\nabla u|^{p(x)}}{p(x)}\,{\rm {\rm dx}}+\int_\Omega \frac{|x|^k|u|^{p(x)}}{p(x)}\,{\rm {\rm dx}}.$$
The two definitions agree up to equivalence of norms (see \cite[Lemma 3.1.6]{Variable_book}).

From \cite[Definition 2.1.1]{Variable_book}, we can verify that $\rho_{\X}$ and $\tilde\rho_{\X}$ are modular on $\X$ and satisfy the following two lemmas:
\begin{lemma}\label{modular_ineq_space}
	Let $u \in \X$, then following relations hold
	\begin{enumerate}
		\item[$(i)$]$ \lv u \rv_{{1,p(x)}}<1(=1;>1) \text{ if and only if } \rho_\X(u)<1(=1;>1);$
		\item[$(ii)$] If $\lv u \rv_{{1,p(x)}}>1$, then $\lv u \rv_{{1,p(x)}} ^{p^{-}}\leq\rho_\X(u)\leq\lv u \rv_{{1,p(x)}}^{p^{+}}$ ;
		\item[$(iii)$] If $\lv u \rv_{{1,p(x)}}<1$, then $\lv u \rv_{{1,p(x)}}^{p^{+}}\leq\rho_\X(u)\leq\lv u \rv_{{1,p(x)}}^{p^{-}}.$
	\end{enumerate}
\end{lemma}

\begin{lemma}\label{lem:modular_cgs_space}
	Let $u,u_{m}  \in \X,~m=1,2,3,\cdots $. Then the following statements are equivalent
	\begin{enumerate}
		\item[$(i)$] $\displaystyle{\lim_{m\ra \infty} }\lv u_{m} - u \rv_{{1,p(x)}} =0;$
		\item[$(ii)$] $\displaystyle{\lim_{m\ra \infty}} \rho_\X(u_{m} -u)=0;$
		\item[$(iii)$] $ u_{m} \text{ converges to $u$ in }  \Omega  \text { in measure and } \displaystyle{\lim_{m\ra \infty}} \rho_\X(u_{m})= \rho_\X(u).$
	\end{enumerate}
\end{lemma}

\begin{lemma}[\cite{Ky_Ho_weig}]\label{lm:Sobolev_embedding}
    Let $p\in \mathcal{C}_+(\R^N)$. Then, for any $q(x)\in L^\infty(\R^N),$ the following embedding holds
    \begin{equation*}
        W^{1,p(x)}(\R^N)\hookrightarrow L^{q(x)}(\R^N) 
    \end{equation*}
    where $q(x)$ satisfies $p(x)\le q(x)\ll p^*(x)$ for a.e. $x\in\R^N.$
\end{lemma}
By standard arguments, it can be shown that $W^{1,p(\cdot)}(|x|^k,\mathbb{R}^N)$ is a reflexive and separable Banach space (see \cite{Ky_Ho_weig,KIM2010}). Let $\X$ denote the closure of $C_c^\infty(\mathbb{R}^N)$ in $W^{1,p(x)}(|x|^k,\mathbb{R}^N)$. For simplicity, we write $\|\cdot\|_\X$ in place of $\|\cdot\|_{1,p(x)}$ throughout.
\begin{remark}
From the definition of the space $W^{1,p(x)}(|x|^k,\om)$ and Lemma \ref{lm:Sobolev_embedding}, we naturally get the following embedding 
\begin{equation}\label{embedding}
    \X\hookrightarrow W^{1,p(x)}(|x|^k,\R^N)\hookrightarrow L^{p(x)}(\R^N).
\end{equation}
\end{remark}
\begin{lemma}[\cite{Kopaliani_GN}]
Let $1 \le p^{-} \le p^{+} < \infty$ and $1 \le q^{-} \le q^{+} < \infty$. Then for $0 < \alpha < 1$
\[
\big(L^{p(\cdot)}(\mathbb{R}^N)\big)^{1-\alpha}
\big(L^{q(\cdot)}(\mathbb{R}^N)\big)^{\alpha}
= L^{s(\cdot)}(\mathbb{R}^N),
\]
where
\[
\frac{1}{s(t)} = \frac{1-\alpha}{p(t)} + \frac{\alpha}{q(t)}, \quad t \in \mathbb{R}^N.
\]
\end{lemma}
We observe that by taking $q(x) = p^*(x):= \frac{Np(x)}{N-p(x)}$ and applying the Sobolev embedding theorem for variable exponent Sobolev spaces, we obtain the following inequality.
\begin{lemma}[{\bf Gagliardo-Nirenberg inequality}]\label{GN_inequality}
    Suppose that $1\le p(x)<s(x)\ll p^*(x)$ and $u\in W^{1,p(x)}(\R^N).$ Then there exists a positive constant $K_\alpha$ such that
    $$\|u\|_{L^{s(x)}(\R^N)}\le K_\al\|u\|^{1-\al}_{L^{p(x)}(\R^N)}\|\na u\|_{L^{p(x)}(\R^N)}^\al,$$
    where
    $$\al:=N\lb\frac{1}{p(x)}-\frac{1}{s(x)}\rb=\frac{N(s(x)-p(x))}{s(x)p(x)}.$$
\end{lemma}
Next, we prove the following crucial compact embedding theorem related to the space $\X.$
\begin{lemma}
   Assume that $p \in \mathcal{C}_+(\mathbb{R}^N)$. Then the compact embedding
\[
\mathcal{X} \hookrightarrow\hookrightarrow L^{q(x)}(\mathbb{R}^N)
\]
holds for all $q(x)$ such that $p(x) \le q(x) \ll p^*(x)$.
\end{lemma}
\begin{proof}
    We first prove the result for $q(x)=p(x).$

    Let $\{u_n\}\subset\X$ be a sequence such that
   \begin{equation}\label{wk_cgt}
       u_n \weak 0 \quad\text{ weakly in }\X
   \end{equation}
    which implies that 
        $$u_n \weak 0 \quad\text{ weakly in }L^{p(x)}(\R^N).$$
Moreover, we have $M:=\sup_n \|u_n\|_\X<\infty.$ For any $\ve>0$, there exists $\sigma>0$ such that $|x|^{-k}\le\ve$ whenever $|x|\ge\sigma.$
From \eqref{wk_cgt}, we obtain 
\begin{equation*}
    u_n \ra 0 \quad\text{ strongly in }L^{p(x)}(B_\sigma),
\end{equation*}
and consequently, using the properties of the modular function (see \cite[Theorem 1.4]{fan_lp_wp_spaces}), there exists $m\in \N$ such that
$$\int_{B_\sigma}|u_n|^{p(x)}\,{\rm {\rm dx}}\le \ve \quad\forall\, n\ge m.$$
Hence, for $n\ge m,$ we obtain
\begin{align*}
    \int_{\R^N}|u_n|^{p(x)}\,{\rm {\rm dx}}&=\int_{\R^N\sm B_\sigma}|u_n|^{p(x)}\,{\rm {\rm dx}}+\int_{B_\sigma}|u_n|^{p(x)}\,{\rm {\rm dx}}\\
    &\le \int_{\R^N\sm B_\sigma}|x|^{-k}|x|^k|u_n|^{p(x)}\,{\rm {\rm dx}}+\ve\\
    &\le \ve\int_{\R^N\sm B_\sigma}|x|^k|u_n|^{p(x)}\,{\rm {\rm dx}}+\ve\\
    &\le \ve C(M),
\end{align*}
where $C(M)$ does not depends on $n$. 

Thus, $u_n\ra0$ strongly in $L^{p(x)}(\R^N),$ and we conclude that $$\X\cts\cts L^{p(x)}(\R^N).$$

We now consider the case $p(x)<q(x)<p^*(x).$ By Lemma \ref{GN_inequality}, we have
\begin{align*}
    \|u_n\|_{L^{q(x)}(\R^N)}&\le K_\al\|u_n\|^{1-\al}_{L^{p(x)}(\R^N)}\|\na u_n\|_{L^{p(x)}(\R^N)}^\al\\
    &\le C \|u_n\|^{1-\al}_{L^{p(x)}(\R^N)}\|u_n\|_\X^\al.
\end{align*}
Since \( u_n \to 0 \) in \( L^{p(x)}(\mathbb{R}^N) \) and \( {u_n} \) is bounded in \( \mathcal{X} \), it follows that \( u_n \to 0 \) strongly in \( L^{q(x)}(\mathbb{R}^N) \). This completes the proof of the lemma.
\end{proof}
\begin{definition}
		We say that $u\in \X$ is a weak solution to $\eqref{problem}$ if $u\in S(c)$ and satisfies
		\begin{align*}
			\int_{\rnn}|\grad u|^{p(x)-2}\na u\na\varphi\,{\rm {\rm dx}}&+\int_{\rnn}|x|^k|u|^{p(x)-2}u\varphi\,{\rm {\rm dx}}\nonumber\\
			&-\lambda\int_{\rnn} |u|^{p(x)-2}u \varphi\,{\rm {\rm dx}}-\int_{\rnn}|u|^{q(x)-2}u\varphi\,{\rm {\rm dx}}=0
		\end{align*}
		$\text{for all } \,\varphi \in \X$.
\end{definition}
\begin{definition}
    Let $c> 0$ be arbitrary, we say that $u \in S_c$ is a \emph{ground state} if 
	\[
	E'|_{S_c}(u) = 0 \quad \text{and} \quad E(u) = \inf\{ E(v) \,:\, v \in S_c,\, E'|_{S_c}(v) = 0 \}.
	\]
 \end{definition}
	That is, $u\in\X$ is a ground state of \eqref{problem} if it is the
	least energy solution among all nontrivial weak solutions of \eqref{problem}.

    With these preliminaries in place, we can now state the main results of this article. We begin with a regularity result for weak solutions.
    \begin{theorem}\label{lemma_regularity}
    Let us assume that $p(x)$ satisfies $(p_H)$ and $u\in\X$ be a weak solution of \eqref{problem} then $u\in C^{1,\alpha}_{loc}(\rnn)$ for some $\alpha\in(0,1)$.
\end{theorem}
Next, we establish a Poho\v zaev-type identity adapted to the variable exponent setting in $\R^N$.
\begin{theorem}\label{Pohozaev_variable}
Let $u\in \X$ be a weak solution of \eqref{problem}, then it satisfies the following
\begin{eqnarray}\label{pohozaev}
   && \int_{\mathbb{R}^N}\left(\frac{N-p(x)}{p(x)}\right)|\nabla u|^{p(x)}{\rm dx}+\int_{\mathbb{R}^N}\lb\frac{N+k}{p(x)}\rb|x|^k|u|^{p(x)}{\rm dx}-N\lambda\int_{\R^N}\frac{|u|^{p(x)}}{p(x)}{\rm dx}\nonumber\\
   && = \int_{\R^N}\left(\ln|u|-\frac{1}{q(x)}\right)|u|^{q(x)}\frac{(x\cdot\nabla q(x))}{q(x)}{\rm dx} + N\int_{\R^N}\frac{|u|^{q(x)}}{q(x)}{\rm dx}\nonumber\\
   && \;\;\;\;\;   -\int_{\R^N}\left(\left(\ln|\nabla u|-\frac{1}{p(x)}\right)|\nabla u|^{p(x)}{\rm dx}
   +\left(\ln{|u|-\frac{1}{p(x)}}\right)|u|^{p(x)}|x|^k\right)\frac{(x\cdot\nabla p(x))}{p(x)}{\rm dx}\nonumber\\
   && \;\;\;\;\;+\int_{\R^N}\lambda\left(\ln|u|-\frac{1}{p(x)}\right)|u|^{p(x)}\frac{(x\cdot\nabla p(x))}{p(x)}{\rm dx}.
    \end{eqnarray}
\end{theorem}
Finally, using Theorems \ref{lemma_regularity} and \ref{Pohozaev_variable}, we establish the following existence result.
\begin{theorem}\label{main_thm}
Assume that $p(x)$ satisfies $(p_H)$ and $k>0$. Then the following hold 
	\begin{enumerate}
		\item[$(i)$] If $q(x)$ satisfies 
        \begin{equation}\label{condition_on_q1}
        p^+ + \frac{(p^+)^2}{N} < \inf_{x\in \rnn}q(x) \ll p^*(x),
        \end{equation}
        then for any $\sigma > 0$, there exists $c_0 = c_0(\sigma) > 0$ such that, for every $0 < c < c_0$, 
		the infimum $\gamma_c$ is attained by some $u_c \in S(c) \cap B_\sigma$, which solves problem \eqref{problem}
		with a Lagrange multiplier $\lambda=\lambda_{c}\in\mathbb{R}$.
		
		\item[$(ii)$]  In addition, if $p(x)\in\mathcal{P}$ and 
    \begin{equation}\label{condition_on_q2}
        2p^{+}-p^{-}+\frac{p^{+}p^{-}}{N}
        <\inf_{x\in \rnn}q(x)\ll p^{*}(x),
    \end{equation}
    then there exists $0 < c_* < c_0$ such that, for all $0 < c < c_*$, 
		the solution $u_c$ obtained in $(i)$ is a ground state to problem~\eqref{problem}
		with some $\lambda = \lambda_c \in \mathbb{R}$. Moreover, it satisfies $E(u_c) = \gamma_c> 0$, and
		\[
        u_c\to 0 \quad \text{in }\mathcal{X}\text{ as } c\to 0.
		\]
	\end{enumerate}
\end{theorem}
\begin{remark}
From \eqref{condition_on_q1} and \eqref{condition_on_q2} we can observe that,
in the constant exponent case, i.e., $p(x)\equiv p$ and $q(x)\equiv q$,
both conditions reduce to
\[
    p+\frac{p^{2}}{N} < q < p^{*}.
\]
This is exactly the $L^{p}$-supercritical range (see \cite{Normalised_potential,Normalised_p_laplacian}). 
Hence, the different expressions appear simply because $p(x)$ and $q(x)$ are variable exponents, and these results
remain consistent with the classical constant exponent case.
\end{remark}

\begin{remark}
The stronger condition on $p(x)$ and $q(x)$ in part $(ii)$ are needed to show that the solution becomes a ground state. 
The proof of part $(ii)$ relies on a Poho\v zaev-type identity and
energy comparisons that are not available under the weaker assumption of
$(i)$. In particular, the sharper lower bound on $\inf q(x)$ ensures
positivity of the energy level $\gamma_{c}$ for small $c$, which is then used
to show that $u_{c}$ is a ground state.
\end{remark}

\begin{remark}
    From Lemma \ref{lemma_separation}, we can observe that for $q^+\leq p^- + \dfrac{p^+p^-}{N}$, using Gagliardo-Nirenberg inequality, the functional $E(u)$ is coercive and bounded below on $S(c)$. Therefore, the classical variational approach can be used to prove the existence of a weak solution.
\end{remark}

\section{Regularity Results}\label{reg}
Let us start our investigation by studying the regularity properties of a weak solution to \eqref{problem}. The following proposition will be used to determine the local boundedness of the solution:
\begin{proposition}\label{prop_2.1}
    Let $u\in W^{1,p}(\Omega)$, assume that for any $B_{\sigma}\subset \subset \Omega$ with $0<\sigma <R_0$, $\delta\in (0,1)$ and $k\geq k_0 >0$,
    $$\int_{A_{k,\sigma(1-\delta)}}|\nabla u|^p{\rm {\rm dx}}\leq c\left[\int_{A_{k,\sigma}}\left|\frac{u-k}{\delta \sigma}\right|^{p^*}{\rm {\rm dx}}+(k^r+1)|A_{k,\sigma}|\right],$$
    where $0<r\leq p^*$, $p^*=\frac{Np}{N-p}$ is the Sobolev critical exponent, $c>0$ is a constant and $A_{k,\sigma}=\big\{x\in B_{\sigma}:u(x)>k\big\}$. Then $u$ is locally bounded above in $\Omega$.
\end{proposition}
A detailed proof of the above proposition can be found in \cite[Lemma~2.5]{fusco1993some}.
\begin{lemma}
    Suppose $u\in \X$ solves \eqref{problem} weakly, then $u\in L^{\infty}_{loc}(\mathbb{R}^N)$. 
\end{lemma}
\begin{proof}
    Let $\Omega\subset \mathbb{R}^N$ be any bounded domain with smooth boundary. For any $x_0\in \Omega$, let $R>0$ be such that $B_{R}(x_0) \subset \Omega$. Set 
    $$\bar{p}_+:=\max_{x\in B_R(x_0)}p(x);\;\;\bar{p}_-:=\min_{x\in B_R(x_0)}p(x);\;\;\bar{q}_+:=\max_{x\in B_R(x_0)}q(x) \text{ and } \bar{q}_-:=\min_{x\in B_R(x_0)}q(x).$$
    Clearly, $\bar{p}_-< N$; and by continuity of $p$ and $q$, we can find $R>0$ small enough such that $q(x)<p^*(x')$ for all $x,x'\in B_R(x_0)$, thus $$\frac{Nq(x)}{N+q(x)}<p(x') \text{ for all } x,x'\in B_R(x_0)$$ and hence 
    $$\bar{p}_+\leq \bar{q}_+< \bar{p}^*_-:= \frac{N \bar{p}_-}{N-\bar{p}_-}.$$
    For any $y\in B_R(x_0)$, let $0<t\leq s$ be such that $\bar{B}_t(y)\subset B_s(y)\subset B_R(x_0)$ and $0<|s-t|<1$. Taking $\phi\in C_c^{\infty}(\mathbb{R}^N)$ with 
    $$\left\{ 
    \begin{array}{cc}
       {\rm Supp}(\phi)\subset B_s(y);  & 0\leq \phi(x) \leq 1 \text{ for all } x;\\
        |\nabla \phi |\leq \frac{2}{s-t};  & \text{ and } \phi(x)=1 \text{ in } B_t(y),
    \end{array}
    \right.$$
    we define our test function $\psi_j:=\phi^{\bar{p}_+}(u-j)^+$ for some $j\geq 1$. Thus, we get
    $$\int_{\rnn}|\nabla u|^{p(x)-2}\nabla u\nabla \psi_j{\rm {\rm dx}} + \int_{\rnn}|x|^k|u|^{p(x)-2}u \psi_j{\rm {\rm dx}}=\lambda \int_{\rnn}|u|^{p(x)-2}u\psi_j{\rm {\rm dx}}+\int_{\rnn}|u|^{q(x)-2}u\psi_j{\rm {\rm dx}},$$
    denoting $A_{j,s}:=\{x\in B_s(y): u(x)\geq j\}$, the above equation becomes:
    \begin{eqnarray}\label{eq_2.1}
        \int_{A_{j,s}}|\nabla u|^{p(x)}\phi^{\bar{p}_+}{\rm {\rm dx}} & = & -\bar{p}_+\int_{A_{j,s}}|\nabla u|^{p(x)-2}\phi^{\bar{p}_+-1}(u-j)\nabla u.\nabla \phi {\rm {\rm dx}} \nonumber\\
        && - \int_{A_{j,s}}|x|^j|u|^{p(x)-2}u\phi^{\bar{p}_+}(u-j){\rm {\rm dx}}\nonumber\\
        &&+ \lambda \int_{A_{j,s}} |u|^{p(x)-2}u\phi^{\bar{p}_+}(x)(u-j){\rm {\rm dx}}
        \nonumber\\
        &&
        +\int_{A_{j,s}} |u|^{q(x)-2}u\phi^{\bar{p}_+}(u-j){\rm {\rm dx}}.
    \end{eqnarray}
    Let us estimate each term on the right-hand side of \eqref{eq_2.1}. Firstly, for any $\epsilon>0$, by Young's inequality
    \begin{eqnarray*}
        && -\bar{p}_+\int_{A_{j,s}}|\nabla u|^{p(x)-2}\phi^{\bar{p}_+-1}(u-j)\nabla u.\nabla \phi {\rm {\rm dx}}\\
        && \leq  \bar{p}_+ \int_{A_{j,s}} |\nabla u|^{p(x)-1}\phi^{\bar{p}_+-1}|\nabla \phi||u-j|{\rm {\rm dx}}\\
        & &\leq \bar{p}_+\int_{A_{j,s}}\left(\left( \frac{p(x)-1}{p(x)}\right)\left( \epsilon \phi^{\bar{p}_+-1}|\nabla u|^{p(x)-1}\right)^{\frac{p(x)}{p(x)-1}}+\frac{1}{p(x)}\left(\frac{|\nabla \phi||u-j|}{\epsilon}\right)^{p(x)}\right){\rm {\rm dx}}\\
        && \leq \bar{p}_+\left(\frac{\bar{p}_+-1}{\bar{p}_-}\right) \int_{A_{j,s}}|\nabla u|^{p(x)}\epsilon^{\frac{p(x)}{p(x)-1}}\phi^{\frac{(\bar{p}_+-1)p(x)}{p(x)-1}}{\rm {\rm dx}}+\frac{\bar{p}_+}{\bar{p}_-}\epsilon^{-\bar{p}_+}\int_{A_{j,s}}|\nabla \phi|^{p(x)}|u-j|^{p(x)}{\rm {\rm dx}}\\
        && \leq \bar{p}_+\left(\frac{\bar{p}_+-1}{\bar{p}_-}\right)\epsilon^{\frac{\bar{p}_-}{\bar{p}_+-1}}\int_{A_{j,s}}|\nabla u|^{p(x)}\phi^{\bar{p}_+}{\rm {\rm dx}}+\left(\frac{\bar{p}_+}{\bar{p}_-}\right)\epsilon^{-\bar{p}_+}\int_{A_{j,s}}\left|\frac{2}{s-t} \right|^{p(x)}|u-j|^{p(x)}{\rm {\rm dx}},
    \end{eqnarray*}
    taking $\epsilon>0$ small enough so that $$\bar{p}_+\left(\frac{\bar{p}_+-1}{\bar{p}_-}\right)\epsilon^{\frac{\bar{p}_-}{\bar{p}_+-1}}<\frac{1}{2}, \text{ that is, } \epsilon<\left(\frac{\bar{p}_-}{2\bar{p}_+(\bar{p}_+-1)}\right)^{\frac{\bar{p}_+-1}{\bar{p}_-}},$$
    we get
    \begin{eqnarray}\label{eq_2.2}
        && -\bar{p}_+\int_{A_{j,s}}|\nabla u|^{p(x)-2}\phi^{\bar{p}_+-1}(u-j)\nabla u.\nabla \phi {\rm {\rm dx}}\nonumber\\
        && \leq 
        \left(\frac{\bar{p}_+}{\bar{p}_-}\right)2^{\bar{p}_+}\epsilon^{-\bar{p}_+}\left(\int_{A_{j,s}\cap\left\{\left|\frac{u(x)-j}{s-t}\right|<1\right\}}\left|\frac{u(x)-j}{s-t}\right|^{p(x)}{\rm {\rm dx}}
        +\int_{A_{j,s}\cap\left\{\left|\frac{u(x)-j}{s-t}\right|\geq 1\right\}}\left|\frac{u(x)-j}{s-t}\right|^{p(x)}{\rm {\rm dx}}\right)\nonumber\\
        && + \frac{1}{2}\int_{A_{j,s}}|\nabla u|^{p(x)}\phi^{\bar{p}_+}{\rm {\rm dx}}\nonumber\\
        && \leq         \left(\frac{\bar{p}_+}{\bar{p}_-}\right)2^{\bar{p}_+}\epsilon^{-\bar{p}_+}\left(|A_{j,s}|+\int_{A_{j,s}}\left| \frac{u-j}{s-t}\right|^{\bar{p}^*_-}{\rm {\rm dx}}\right)+\frac{1}{2}\int_{A_{j,s}}|\nabla u|^{p(x)}\phi^{\bar{p}_-}{\rm {\rm dx}}\nonumber\\
        && \leq \frac{1}{2}\int_{A_{j,s}}|\nabla u|^{p(x)}\phi^{\bar{p}_+} {\rm {\rm dx}}+C_1|A_{j,s}|+C_1\int_{A_{j,s}}\left| \frac{u-j}{s-t}\right|^{\bar{p}^*_-}{\rm {\rm dx}}.
    \end{eqnarray}
    Next, 
    \begin{eqnarray}\label{eq_2.3}
    && \int_{A_{j,s}} |u|^{p(x)-2}u\phi^{\bar{p}_+}(x)(u-j){\rm {\rm dx}}
     \leq \int_{A_{j,s}}|u|^{p(x)-1}|u-j|{\rm dx} \nonumber\\
    & \leq & \int_{A_{j,s}}\frac{|u-j|^{p(x)}}{p(x)}{\rm dx}+\int_{A_{j,s}}\left(\frac{p(x)-1}{p(x)}\right)\left(|u|^{p(x)-1}\right)^{\frac{p(x)}{p(x)-1}}{\rm dx}\nonumber\\
    & \leq & \int_{A_{j,s}\cap \{|u-j|<1\}} \frac{|u-j|^{p(x)}}{\bar{p}_-} {\rm dx}+\int_{A_{j,s}\cap \{|u-j|\geq 1\}}\frac{|u-j|^{\bar{p}^*_-}}
    {\bar{p}_-}{\rm dx}+\int_{A_{j,s}}\left(\frac{p(x)-1}{p(x)}\right)|u|^{p(x)}{\rm dx}\nonumber\\
    &\leq & \frac{|A_{j,s}|}{\bar{p}_-}+\frac{1}{\bar{p}_-}\int_{A_{j,s}}\left|\frac{u-j}{s-t} \right|^{\bar{p}^*_-}{\rm dx}+\left(\frac{\bar{p}_+-1}{\bar{p}_-}\right)j^{\bar{p}_+}|A_{j,s}|\nonumber\\
    & \leq & C_2 \int_{A_{j,s}}\left|\frac{u-j}{s-t} \right|^{\bar{p}^*_-}{\rm dx}+C_2(j^{\bar{p}_+}+1)|A_{j,s}|.
    \end{eqnarray}
    Similarly,
    \begin{equation}\label{eq_2.4}
        \int_{A_{j,s}} |u|^{q(x)-2}u\phi^{\bar{p}_+}(x)(u-j){\rm dx} \leq C_3\int_{A_{j,s}}\left|\frac{u-j}{s-t} \right|^{\bar{p}^*_-}{\rm dx}+C_3(j^{\bar{q}_+}+1)|A_{j,s}|,
    \end{equation}
    and since $\Omega$ is a bounded domain, there exists $C>0$ such that $|x|<C$ for all $x\in A_{j,s}$, hence by \eqref{eq_2.3}
    \begin{eqnarray}\label{eq_2.5}
        \int_{A_{j,s}}|x|^j|u|^{p(x)-2}u\phi^{\bar{p}_+}(u-j){\rm dx} & \leq & C^j\int_{A_{j,s}}|u|^{p(x)-1}\phi^{\bar{p}_+}|u-j|{\rm dx}\nonumber\\
         &\leq & C_4 \int_{A_{j,s}}\left|\frac{u-j}{s-t} \right|^{\bar{p}^*_-}{\rm dx}+C_2(j^{\bar{p}_+}+1)|A_{j,s}|.
    \end{eqnarray}
    Using \eqref{eq_2.2}, \eqref{eq_2.3}, \eqref{eq_2.4} and \eqref{eq_2.5} in \eqref{eq_2.1} we get:
    \begin{eqnarray*}
        \int_{A_{j,s}}|\nabla u|^{p(x)}\phi^{\bar{p}_+}{\rm dx} & \leq & 2 C_5 \left(\int_{A_{j,s}}\left| \frac{u-j}{s-t}\right|^{\bar{p}^*_-}{\rm dx}+ (1+j^{\bar{p}_+}+j^{\bar{q}_+})|A_{j,s}|\right)\\
        & \leq & C_6  \left(\int_{A_{j,s}}\left| \frac{u-j}{s-t}\right|^{\bar{p}^*_-}{\rm dx}+ (1+j^{\bar{q}_+})|A_{j,s}|\right),
    \end{eqnarray*}
    hence,
    \begin{eqnarray*}
   \int_{A_{j,t}}|\nabla u|^{\bar{p}_-}{\rm dx} & = & \int_{A_{j,t}\cap \{|\nabla u|<1\}}|\nabla u |^{\bar{p}_-}{\rm dx}+\int_{A_{j,t}\cap \{|\nabla u|\geq 1\}}|\nabla u|^{\bar{p}_-}{\rm dx}\\
   &\leq & |A_{s,t}|+ \int_{A_{j,t}}|\nabla u|^{p(x)}{\rm dx}\\
& \leq & |A_{j,s}|+\int_{A_{j,t}}|\nabla u|^{p(x)}\phi^{\bar{p}_+}{\rm dx}+\int_{A_{j,s}\setminus A_{j,t}}|\nabla u|^{p(x)}\phi^{\bar{p}_+}{\rm dx}\nonumber\\
       & = & \int_{A_{j,s}}|\nabla u|^{p(x)}\phi^{\bar{p}_+}{\rm dx}+|A_{j,s}|\nonumber\\
       & \leq & C_6  \left(\int_{A_{j,s}}\left| \frac{u-j}{s-t}\right|^{\bar{p}^*_-}{\rm dx}+ (1+j^{\bar{q}_+})|A_{j,s}|\right).
    \end{eqnarray*}
    Since, $t<s$, we can replace $t$ by $s(1-\delta)$ for some $\delta\in (0,1)$, hence by Proposition \ref{prop_2.1}, $u$ is locally bounded in $\Omega$. Further, since $u\in W^{1,p(x)}(\mathbb{R}^N)$, the boundedness can be extended to $\Omega$ as done in \cite[Lemma~3.6]{byun2017global}. Thus $u\in L^{\infty}(\Omega)$ and hence $u\in L^{\infty}_{loc}(\mathbb{R}^N)$.
\end{proof}
\begin{proof}[Proof of Theorem \ref{lemma_regularity}]
    Following the approach of \cite[Theorem 1.1]{Fan_regurlarity} (see also 
\cite[Theorem 3.2]{byun2017global}), we obtain the stated regularity result.
\end{proof}

\section{Poho\v{z}aev Identity}\label{pohoz}
In this section, we provide a proof of Theorem \ref{Pohozaev_variable}. This Poho\v zaev-type identity plays an important role in the characterization of ground state solutions.
\begin{proof}[Proof of Theorem \ref{Pohozaev_variable}]
	Let $\phi\in C_c^{\infty}(\mathbb{R}^N)$ be such that $\phi=1$ in $B_1(0)$, $0\leq \phi \leq 1$ and $\phi=0$ in $\mathbb{R}^N\setminus B_2(0)$. Define $v_t(x):=\phi(tx)x\cdot\nabla u$. Clearly, $v_t\in \X$, by regularity of $u$ (see \autoref{lemma_regularity}). Multiplying \eqref{problem} by $v_t$ and integrating it over $\mathbb{R}^N$, we get
    \begin{eqnarray}\label{eq_3.2}
        -\int_{\mathbb{R}^N} \text{div}(|\nabla u|^{p(x)-2}\nabla u)v_t {\rm dx}+\int_{\mathbb{R}^N}|x|^k|u|^{p(x)-2}uv_t {\rm dx}
       & = &\lambda \int_{\R^N}|u|^{p(x)-2}uv_t {\rm dx}\nonumber\\
       &&+\int_{\R^N}|u|^{q(x)-2}uv_t{\rm dx}.
    \end{eqnarray}
	Now, by the divergence theorem, we have
	\begin{eqnarray*}
		0 & = & \int_{\partial \mathbb{R}^N} \frac{\phi(tx)|u|^{p(x)}}{p(x)}x\cdot\nu d\sigma = \int_{\mathbb{R}^N}\nabla\cdot\left(\frac{\phi(tx)|u|^{p(x)}x}{p(x)}\right){\rm {\rm dx}}\\
		& = & \sum_{i=1}^{N}\int_{\mathbb{R}^N}\frac{\phi(tx)|u|^{p(x)}}{p(x)}{\rm {\rm dx}}+\sum_{i=1}^{N}\int_{\mathbb{R}^N}x_i\frac{\partial}{\partial x_i}\left(\frac{\phi(tx)|u|^{p(x)}}{p(x)}\right){\rm {\rm dx}}\\
		& = & N\int_{\mathbb{R}^N}\frac{\phi(tx)|u|^{p(x)}}{p(x)}+\int_{\mathbb{R}^N}\phi(tx)|u|^{p(x)-2}u(x)(x.\nabla u){\rm {\rm dx}}+ t\int_{\mathbb{R}^N}\frac{|u|^{p(x)}}{p(x)}(x\cdot\nabla \phi(tx)){\rm {\rm dx}}\\
		&& + \int_{\mathbb{R}^N}\frac{\phi(tx)|u|^{p(x)}}{p(x)}\left(\ln|u|-\frac{1}{p(x)}\right)x\cdot\nabla p(x){\rm {\rm dx}},
	\end{eqnarray*}
	thus, 
	\begin{eqnarray*}
		\int_{\mathbb{R}^N}|u|^{p(x)-2}uv_t & = &-\int_{\mathbb{R}^N}(N\phi(tx)+tx\cdot\nabla \phi(tx))\frac{|u|^{p(x)}}{p(x)}\\
		&& -\int_{\mathbb{R}^N}\left(\ln|u|-\frac{1}{p(x)}\right)\frac{\phi(tx)|u|^{p(x)}}{p(x)}x\cdot\nabla p(x){\rm {\rm dx}}.
	\end{eqnarray*} 
By the dominated convergence theorem, we get
	\begin{equation}\label{eq_3.3}
		\lim_{t\rightarrow 0} \int_{\mathbb{R}^N}|u|^{p(x)-2}uv_t  =  -N\int_{\mathbb{R}^N}\frac{|u|^{p(x)}}{p(x)}{\rm {\rm dx}}- \int_{\mathbb{R}^N}\left(\ln|u|-\frac{1}{p(x)}\right)\frac{|u|^{p(x)}}{p(x)}x\cdot\nabla p(x){\rm {\rm dx}},
	\end{equation}
	similarly,
	\begin{equation}\label{eq_3.4}
	\lim_{t\rightarrow 0} \int_{\mathbb{R}^N}|u|^{q(x)-2}uv_t  =  -N\int_{\mathbb{R}^N}\frac{|u|^{q(x)}}{q(x)}{\rm {\rm dx}}- \int_{\mathbb{R}^N}\left(\ln|u|-\frac{1}{q(x)}\right)\frac{|u|^{q(x)}}{q(x)}x\cdot\nabla q(x){\rm {\rm dx}}.
\end{equation}	
 Again, using the divergence theorem, we have
\begin{eqnarray*}
	0 & = & \int_{\partial \mathbb{R}^N}\frac{\phi(tx)|x|^k|u|^{p(x)}}{p(x)}x.\nu d\sigma = \int_{\mathbb{R}^N} \nabla\cdot\left( \frac{\phi(tx)|x|^k|u|^{p(x)}x}{p(x)}\right){\rm {\rm dx}}\\
	& = & N\int_{\mathbb{R}^N}\frac{\phi(tx)|x|^k|u|^{p(x)}}{p(x)}{\rm {\rm dx}}+\sum_{i=1}^{N}\int_{\mathbb{R}^N}x_i\frac{\partial}{\partial x_i}\left(\frac{\phi(tx)|x|^k|u|^{p(x)}}{p(x)}\right){\rm {\rm dx}}\\
	& = & (N+k)\int_{\mathbb{R}^N}\frac{\phi(tx)|x|^k|u|^{p(x)}}{p(x)}{\rm {\rm dx}} +\int_{\R^N} \phi(tx)|x|^k|u|^{p(x)-2}u(x\cdot\nabla u){\rm {\rm dx}}\\
	&& +t \int_{\mathbb{R}^N}\frac{|x|^k|u|^{p(x)}}{p(x)}(x\cdot\nabla \phi(tx)){\rm {\rm dx}} + \int_{\mathbb{R}^N}\frac{\phi(tx)|x|^k|u|^{p(x)}}{p(x)}\left(\ln|u|-\frac{1}{p(x)}\right)(x\cdot\nabla p(x)){\rm {\rm dx}},
\end{eqnarray*}
thus
\begin{eqnarray*}
	\int_{\mathbb{R}^N}|x|^k|u|^{p(x)-2}uv_t & = & -\int_{\mathbb{R}^N} \left((N+k)\phi(tx)+tx\cdot\nabla\phi(tx)\right)\frac{|x|^k|u|^{p(x)}}{p(x)}{\rm {\rm dx}}\\
	&& -\int_{\mathbb{R}^N}\left(\ln|u|-\frac{1}{p(x)}\right)\frac{\phi(tx)|x|^k|u|^{p(x)}}{p(x)}(x\cdot\nabla p(x)){\rm {\rm dx}},
\end{eqnarray*}
 and hence, by the dominated convergence theorem, we get: 
 \begin{eqnarray}\label{eq_3.5}
    \lim_{t\rightarrow 0}\int_{\mathbb{R}^N}|x|^k|u|^{p(x)-2}uv_t {\rm dx} & = &  -(N+k)\int_{\mathbb{R}^N}\frac{|x|^k|u|^{p(x)}}{p(x)}{\rm dx}\nonumber\\
    &&-\int_{\mathbb{R}^N}\left(\ln|u|-\frac{1}{p(x)}\right)\frac{|x|^k|u|^{p(x)}}{p(x)}x\cdot\nabla p(x){\rm dx}.
 \end{eqnarray}
 Next, we will estimate the first term in the LHS of \eqref{eq_3.2}. Since $\phi\in C_c^{\infty}(\mathbb{R}^N)$, we have
 \begin{eqnarray*}
 	0 & = & \int_{\partial \mathbb{R}^N}\phi(tx)|\nabla u|^{p(x)-2}(x\cdot\nabla u)\nabla u\cdot\nu d\sigma = \int_{\mathbb{R}^N}\nabla \cdot(\phi(tx)|u|^{p(x)-2}(x\cdot\nabla u)\nabla u){\rm dx}\\
 	& = & \int_{\mathbb{R}^N}\sum_{i,j=1}^N \frac{\partial}{\partial x_i}\left(\phi(tx)x_j|\nabla u|^{p(x)-2}\frac{\partial u}{\partial x_j}\frac{\partial u}{\partial x_i}\right){\rm dx},
 \end{eqnarray*}
 thus,
 \begin{eqnarray}\label{eq_3.6}
 	&&\int_{\mathbb{R}^N}\text{div}( |\nabla u|^{p(x)-2}\nabla u)v_t {\rm dx}\nonumber\\
 	 && =  \int_{\mathbb{R}^N}\sum_{i,j=1}^N \frac{\partial}{\partial x_i}\left(\phi(tx)x_j|\nabla u|^{p(x)-2}\frac{\partial u}{\partial x_j}\frac{\partial u}{\partial x_i}\right){\rm dx}
 	-\int_{\mathbb{R}^N}\sum_{i,j=1}^N|\nabla u|^{p(x)-2}\frac{\partial u}{\partial x_i}\frac{\partial}{\partial x_i}\left(\phi(tx)x_j \frac{\partial u}{\partial x_j}\right){\rm dx}\nonumber\\
 	& & = -\int_{\mathbb{R}^N}\sum_{i,j=1}^N|\nabla u|^{p(x)-2}\frac{\partial u}{\partial x_i}\frac{\partial}{\partial x_i}\left(\phi(tx)x_j \frac{\partial u}{\partial x_j}\right){\rm dx}\nonumber\\
 	& & = -\sum_{i,j=1}^N \int_{\mathbb{R}^N}|\nabla u|^{p(x)-2}\phi(tx)x_j \frac{\partial u}{\partial x_i}\frac{\partial^2 u}{\partial x_i \partial x_j}{\rm dx}-\sum_{i,j=1}^N\int_{\mathbb{R}^N}|\nabla u|^{p(x)-2}\phi(tx)\frac{\partial u }{\partial x_i}\frac{\partial u}{\partial x_j}\delta_{i,j}{\rm dx}\nonumber\\
 	&& \;\;\;\;-t\sum_{i,j=1}^N\int_{\mathbb{R}^N}x_j|\nabla u|^{p(x)-2}\frac{\partial u}{\partial x_i}\frac{\partial u }{\partial x_j}\frac{\partial \phi(tx)}{\partial x_i}{\rm dx}\nonumber\\
 	&& = -\sum_{i,j=1}^N \int_{\mathbb{R}^N}|\nabla u|^{p(x)-2}\phi(tx)x_j \frac{\partial u}{\partial x_i}\frac{\partial^2 u}{\partial x_i \partial x_j}{\rm dx}-\int_{\mathbb{R}^N}|\nabla u|^{p(x)}\phi(tx){\rm dx}\nonumber\\
 	&&\;\;\;\; -t\sum_{i,j=1}^N\int_{\mathbb{R}^N}x_j|\nabla u|^{p(x)-2}\frac{\partial u}{\partial x_i}\frac{\partial u }{\partial x_j}\frac{\partial \phi(tx)}{\partial x_i}{\rm dx}
 \end{eqnarray}
 Now, multiplying \eqref{problem} by $\phi(tx)u$ and integrating it over $\mathbb{R}^N $, we get
 \begin{eqnarray*}
 	&& -\lambda \int_{\mathbb{R}^N}|u|^{p(x)}\phi(tx) {\rm dx}-\int_{\mathbb{R}^N}|u|^{q(x)}\phi(tx){\rm dx}+\int_{\mathbb{R}^N}\phi(tx)|x|^k|u|^{p(x)}{\rm dx}\\
 	&& = \int_{\mathbb{R}^N} \text{div}(|\nabla u|^{p(x)-2}\nabla u)\phi(tx)u {\rm dx}=\sum_{i=1}^N \int_{\mathbb{R}^N}\frac{\partial }{\partial x_i}\left(|\nabla u|^{p(x)-2}\frac{\partial u}{\partial x_i}\right)\phi(tx)u{\rm dx}\\
 	&& = \sum_{i=1}^N \int_{\mathbb{R}^N}\frac{\partial}{\partial x_i}\left(|\nabla u|^{p(x)-2}\frac{\partial u}{\partial x_i}\phi(tx)u\right){\rm dx}-\sum_{i=1}^N \int_{\mathbb{R}^N}|\nabla u|^{p(x)-2}\frac{\partial u}{\partial x_i}\frac{\partial (\phi(tx)u)}{\partial x_i}{\rm dx}\\
 	&&= -\sum_{i=1}\int_{\mathbb{R}^N}|\nabla u|^{p(x)-2}\frac{\partial u}{\partial x_i}\left(\phi(tx)\frac{\partial u}{\partial x_i}+ut\frac{\partial \phi(tx)}{\partial x_i}\right){\rm dx}\\
 	&&= -\int_{\mathbb{R}^N}|\nabla u|^{p(x)}\phi(tx){\rm dx}-t \sum_{i=1}^N\int_{\mathbb{R}^N}|\nabla u|^{p(x)-2}\frac{\partial u}{\partial x_i}\frac{\partial \phi(tx)}{\partial x_i}u {\rm dx},
 \end{eqnarray*}
 thus,
 \begin{eqnarray}\label{eq_3.7}
 	\int_{\mathbb{R}^N}|\nabla u|^{p(x)}\phi(tx) {\rm dx} & = & \lambda \int_{\mathbb{R}^N}|u|^{p(x)}\phi(tx){\rm dx} +\int_{\mathbb{R}^N}|u|^{q(x)}\phi(tx){\rm dx}-\int_{\mathbb{R}^N}\phi(tx)|x|^k|u|^{p(x)}{\rm dx}\nonumber\\
 	&& -t \sum_{i=1}^N\int_{\mathbb{R}^N}|\nabla u|^{p(x)-2}\frac{\partial u}{\partial x_i}\frac{\partial \phi(tx)}{\partial x_i}u {\rm dx}.
 \end{eqnarray}
 Using \eqref{eq_3.7} in \eqref{eq_3.6}, we get
 \begin{eqnarray}\label{eq_3.8}
 	&&\int_{\mathbb{R}^N}\text{div}(|\nabla u|^{p(x)-2}\nabla u)v_t {\rm dx}\nonumber\\
 	&& =  -\sum_{i,j=1}^N \int_{\mathbb{R}^N}|\nabla u|^{p(x)-2}\phi(tx)x_j \frac{\partial u}{\partial x_i}\frac{\partial^2 u}{\partial x_i \partial x_j}{\rm dx}
 	-\lambda \int_{\mathbb{R}^N}|u|^{p(x)}\phi(tx) {\rm dx}-\int_{\mathbb{R}^N}|u|^{q(x)}\phi(tx){\rm dx}\nonumber\\
 	&&\;\;\;\;+\int_{\mathbb{R}^N}\phi(tx)|x|^k|u|^{p(x)}{\rm dx}
 	 +t \sum_{i=1}^N\int_{\mathbb{R}^N}|\nabla u|^{p(x)-2}\frac{\partial u}{\partial x_i}\frac{\partial \phi(tx)}{\partial x_i}u {\rm dx}\nonumber\\
 	&& \;\;\;\;-t\sum_{i,j=1}^N\int_{\mathbb{R}^N}x_j|\nabla u|^{p(x)-2}\frac{\partial u}{\partial x_i}\frac{\partial u }{\partial x_j}\frac{\partial \phi(tx)}{\partial x_i}{\rm dx}.
 \end{eqnarray}
 Now, since 
 \begin{eqnarray*}
 	&&\frac{\partial}{\partial x_j}\left(\frac{\phi(tx)|\nabla u|^{p(x)}x_j}{p(x)}\right) \\
 	&& = \phi(tx)\left(\frac{|\nabla u|^{p(x)}}{p(x)}\right)+tx_j\left(\frac{|\nabla u|^{p(x)}}{p(x)}\right)\frac{\partial \phi(tx)}{\partial x_j}-x_j\phi(tx)\left(\frac{|\nabla u|^{p(x)}}{p(x)^2}\right)\frac{\partial p(x)}{\partial x_j}\\
 	&& \;\;\;\;+\phi(tx)x_j|\nabla u|^{p(x)-2}\sum_{i=1}^N\left(\frac{\partial u}{\partial x_i}\right)\left(\frac{\partial^2 u}{\partial x_j \partial x_i}\right) +|\nabla u|^{p(x)}\ln|\nabla u|\left(\frac{\phi(tx)x_j}{p(x)}\right)\frac{\partial p(x)}{\partial x_j}.
 \end{eqnarray*}
 Thus, we get
 \begin{eqnarray*}
 	&& \sum_{i,j=1}^N \int_{\mathbb{R}^N}\phi(tx)x_j \frac{\partial u}{\partial x_i}\frac{\partial^2 u}{\partial x_i \partial x_j}|\nabla u|^{p(x)-2}{\rm dx}\\
 	&& = \int_{\mathbb{R}^N}\nabla\cdot\left(\frac{\phi(tx)|\nabla u|^{p(x)}}{p(x)}x\right){\rm dx}-N \int_{\mathbb{R}^N}\frac{\phi(tx)|\nabla u|^{p(x)}}{p(x)}{\rm dx}-t\int_{\mathbb{R}^N}\frac{|\nabla u|^{p(x)}}{p(x)}(x\cdot\nabla \phi(tx)){\rm dx}\\
 	&& \;\;\;\;+\int_{\mathbb{R}^N}\frac{\phi(tx)|\nabla u|^{p(x)}}{p(x)^2}(x\cdot\nabla p(x)){\rm dx}-\int_{\mathbb{R}^N}\left(\frac{\phi(tx)|\nabla u|^{p(x)}}{p(x)}\right)\ln|\nabla u| (x\cdot\nabla p(x)){\rm dx},
 \end{eqnarray*}
 hence, \eqref{eq_3.8} becomes
 \begin{eqnarray*}
 	&& \int_{\mathbb{R}^N}\text{div}(|\nabla u|^{p(x)-2}\nabla u)v_t {\rm dx}\\
 	&& = -\int_{\mathbb{R}^N}\nabla\cdot\left(\frac{\phi(tx)|\nabla u|^{p(x)}}{p(x)}x\right){\rm dx}+N \int_{\mathbb{R}^N}\frac{\phi(tx)|\nabla u|^{p(x)}}{p(x)}{\rm dx}+t\int_{\mathbb{R}^N}\frac{|\nabla u|^{p(x)}}{p(x)}(x\cdot\nabla \phi(tx)){\rm dx}\\
 	&& \;\;\;\;-\int_{\mathbb{R}^N}\left(\frac{1}{p(x)}-\ln|\nabla u|\right)\frac{\phi(tx)|\nabla u|^{p(x)}}{p(x)}(x\cdot\nabla p(x)){\rm dx} 	-\lambda \int_{\mathbb{R}^N}|u|^{p(x)}\phi(tx){\rm dx}\\
 	&& \;\;\;\;-\int_{\mathbb{R}^N}|u|^{q(x)}\phi(tx){\rm dx}+\int_{\mathbb{R}^N}\phi(tx)|x|^k|u|^{p(x)}{\rm dx}
 	+t \int_{\mathbb{R}^N}|\nabla u|^{p(x)-2}u\nabla u\cdot\nabla \phi(tx){\rm dx}\nonumber\\
 	&& \;\;\;\;-t\sum_{i,j=1}^N\int_{\mathbb{R}^N}x_j|\nabla u|^{p(x)-2}\frac{\partial u}{\partial x_i}\frac{\partial u }{\partial x_j}\frac{\partial \phi(tx)}{\partial x_i}{\rm dx}
 \end{eqnarray*}
 Therefore, the Lebesgue dominated convergence theorem gives us the following
 \begin{eqnarray}\label{eq_3.9}
 &&	 \lim_{t\rightarrow 0}\int_{\mathbb{R}^N}\text{div}(|\nabla u|^{p(x)-2}\nabla u)v_t {\rm dx}\nonumber\\
&& = N \int_{\mathbb{R}^N}\frac{|\nabla u|^{p(x)}}{p(x)}{\rm dx}-\lambda \int_{\mathbb{R}^N}|u|^{p(x)}{\rm dx}-\int_{\mathbb{R}^N}|u|^{q(x)} {\rm dx}+\int_{\mathbb{R}^N}|x|^k|u|^{p(x)}{\rm dx}\nonumber\\
 	&&\;\;\;\; -\int_{\mathbb{R}^N}\left(\frac{1}{p(x)}-\ln|\nabla u|\right)\frac{|\nabla u|^{p(x)}}{p(x)}(x\cdot\nabla p(x)){\rm dx}.
 \end{eqnarray}
 Now, taking the limit $t\rightarrow 0$ in \eqref{eq_3.2} and using \eqref{eq_3.3}, \eqref{eq_3.4}, \eqref{eq_3.5} and \eqref{eq_3.9} we get
 \begin{eqnarray*}
 	&& -N\int_{\mathbb{R}^N}\frac{|\nabla u|^{p(x)}}{p(x)}{\rm dx}+\lambda \int_{\mathbb{R}^N}\left(1+\frac{N}{p(x)}\right)|u|^{p(x)}{\rm dx}+\int_{\mathbb{R}^N}\left(1+\frac{N}{q(x)}\right)|u|^{q(x)}{\rm dx}\nonumber\\
 	&& = \int_{\mathbb{R}^N}\left(1+\frac{N+k}{p(x)}\right)\frac{|x|^k|u|^{p(x)}}{p(x)}{\rm dx}+\int_{\mathbb{R}^N}\left(\ln|u|-\frac{1}{p(x)}\right)\frac{|x|^k|u|^{p(x)}}{p(x)}(x\cdot\nabla p(x)){\rm dx}\nonumber\\
 	&& \;\;\;\;-\int_{\mathbb{R}^N}\left(\frac{1}{p(x)}-\ln|\nabla u|\right)\frac{|\nabla u|^{p(x)}}{p(x)}(x\cdot\nabla p(x)){\rm dx}-\int_{\mathbb{R}^N}\left(\ln|u|-\frac{1}{q(x)}\right)\frac{|u|^{q(x)}}{q(x)}(x\cdot\nabla q(x)){\rm dx}\nonumber\\
 	&& \;\;\;\;-\lambda \int_{\mathbb{R}^N}\left(\ln|u|-\frac{1}{p(x)}\right)\frac{|u|^{p(x)}}{p(x)}(x\cdot\nabla p(x)){\rm dx},
 \end{eqnarray*}
 also, since $u$ solves \eqref{problem}, we have
 $$\int_{\mathbb{R}^N}|\nabla u|^{p(x)}{\rm dx}=\lambda \int_{\mathbb{R}^N}|u|^{p(x)}{\rm dx}+\int_{\mathbb{R}^N}|u|^{q(x)}{\rm dx}-\int_{\mathbb{R}^N}|x|^k|u|^{p(x)}{\rm dx},$$
 thus we get \eqref{pohozaev}. 
\end{proof}
We now simplify the identity for exponents in the class $\mathcal{P}$ (see Definition \ref{exponent}). The remark below explains this assumption.

\begin{remark}\label{error}
    Let us consider the term
	\begin{align*}
		R &= \int_{\mathbb{R}^N} 
		\lambda\left(\ln|u| - \frac{1}{p(x)}\right)|u|^{p(x)}\frac{(x \cdot \nabla p(x))}{p(x)} \,{\rm {\rm dx}}
		+ \int_{\mathbb{R}^N}\left(\ln|u| - \frac{1}{q(x)}\right)|u|^{q(x)}\frac{(x \cdot \nabla q(x))}{q(x)}\,{\rm {\rm dx}} \no \\
		&\quad - \int_{\mathbb{R}^N}
		\left[\left(\ln|\nabla u| - \frac{1}{p(x)}\right)|\nabla u|^{p(x)} 
		+ \left(\ln|u| - \frac{1}{p(x)}\right)|u|^{p(x)}|x|^k\right]
		\frac{(x \cdot \nabla p(x))}{p(x)}\,{\rm {\rm dx}}.
\end{align*}
If we choose $p(x)$ and $q(x)$ such that $x\cdot\nabla p(x)=0$ and 
$x\cdot\nabla q(x)=0$, then clearly $R=0$.  
However, such choices are typically discontinuous.  
For instance, in $\mathbb{R}^2$, the exponent
\[
p(x,y)=\frac{x^2}{x^2+y^2}
\]
satisfies $x\cdot\nabla p(x,y)=0$ for $(x,y)\neq 0$, but is not continuous at the origin.

To overcome this, we introduce the class of exponents $\mathcal{P}(r_{0})$ defined 
in Definition \ref{exponent}, for which $x\cdot\nabla p(x)$ and $x\cdot\nabla q(x)$ 
vanish outside a compact annulus. Under this construction, each term in $R$ is supported in $A(r_{0},2r_{0}) = \{x\in\mathbb{R}^{N} : r_{0}<|x|<2r_{0}\}$. 
A careful estimate then shows that
\begin{equation}\label{cgt_result}
      R \to 0 \qquad \text{as } r_{0}\to 0.
\end{equation}
The detailed proof of this estimate is given in \autoref{lm_est_proof} in the Appendix.
\end{remark}
 \begin{definition}\label{exponent}
Let $r_{0}>0$ and let $p_{0}>0$ be a constant.  
We define $\mathcal P(r_{0})$ as the family of variable exponents 
$\tilde p:\mathbb{R}^{N}\to(1,N)$ of the form
\[
\tilde p(x) = (1-\eta(|x|))\,p_{0} + \eta(|x|)\,p(x),
\]
where:
\begin{itemize}
    \item $p(x)$ is a Lipschitz continuous function such that
    \[
    x\cdot\nabla p(x)=0 \qquad \text{for all } |x|\ge 2r_{0},
    \]
    \item $\eta\in C^{\infty}([0,\infty))$ is a radial cut-off function satisfying
    \[
    \eta(r)=0 \quad \text{for } r\le r_{0}, \qquad
    \eta(r)=1 \quad \text{for } r\ge 2r_{0}, \qquad
    0\le \eta(r)\le 1,
    \]
    and
    \[
    |\eta'(r)| \le \frac{C}{r_{0}}
    \qquad \text{for } r_{0}<r<2r_{0}.
    \]
\end{itemize}
\end{definition}

\section{Existence of the normalized ground state solution}\label{existence}
We begin by proving the following lemma, which ensures that the admissible set in the minimization problem is non-empty.

\begin{lemma}\label{lm_1}
    Let $\sigma>0$, if $0<c\le c_1(\sigma),$ then we have $S(c)\cap B_\sigma\neq\emptyset.$
\end{lemma}
\begin{proof}
    Let us define
    \begin{equation}\label{func}
        \varphi_c(x)=a(c)e^{-\frac{\pi}{p^+}|x|^2}
    \end{equation}
    First, we prove that $\varphi_c(x)\in S(c).$ Consider the function $G:(0,\infty)\to (0,\infty)$
\begin{equation*}
a\longmapsto G(a)=\int_{\mathbb{R}^{N}} \frac{a^{p(x)}}{p(x)}\,e^{-\pi \frac{p(x)}{p^+}|x|^{2}}\,{\rm {\rm dx}}.
\end{equation*}
We prove that for every $c>0$ there exists a unique $a>0$ such that $G(a)=c$. For every $x\in\mathbb{R}^N$, the map $a\mapsto a^{p(x)}$ is strictly increasing on $(0,\infty)$.  
Since the remaining factors in the integrand are positive and independent of $a$, it follows that $G(a)$ is strictly increasing on $(0,\infty)$.

Now, we fix $A>0$. For $a\in[0,A]$ and every $x\in \mathbb{R}^N$,
$$
0 \le \frac{a^{p(x)}}{p(x)}e^{-\pi \frac{p(x)}{p^+}|x|^{2}}
\le \frac{A^{p^+}}{p^-}e^{-\pi\frac{p^-}{p^+}|x|^{2}}.
$$
The function on the right is integrable; therefore, by the dominated convergence theorem, $G$ is continuous on $[0, \infty)$. Hence, $G(a)\to0$ as $a\to 0.$ and $G(a)\to\infty$ as $a\to \infty.$
Therefore, for every $c>0$ there exists a unique $a(c)>0$ satisfying $G(a)=c$ and consequently $\varphi_c(x)\in S(c).$ We observe that for $a$ sufficiently small
\begin{equation}\label{eq4}
\frac{a^{p^+}}{p^+}\int_{\mathbb{R}^{N}}\,e^{-\pi|x|^{2}}\,{\rm {\rm dx}}\le\int_{\mathbb{R}^{N}} \frac{a^{p(x)}}{p(x)}\,e^{-\pi \frac{p(x)}{p^+}|x|^{2}}\,{\rm {\rm dx}}=c,
\end{equation}
that is, $a^{p^+}<cp^+$. Thus, we can conclude that $a<1$ if $c<\frac{1}{p^+}.$

Next, we prove that $\varphi_c(x)\in B_\sigma.$ From \eqref{func}, we have
\begin{align}\label{est_1}
            \rnnn|\na\varphi_c|^{p(x)}\,{\rm {\rm dx}}&=\rnnn a^{p(x)} \lb\frac{2\pi}{p^+}\rb^{p(x)}|x|^{p(x)} e^{-\frac{\pi}{p^+}|x|^2p(x)}\,{\rm {\rm dx}}\nonumber\\
            &\le a^{p^-}\lb\frac{2\pi}{p^+}\rb^{p^+}\rnnn |x|^{p^+} e^{-\frac{\pi}{p^+}|x|^2p^-}\,{\rm {\rm dx}}\no\\
            &\le a^{p^-}\underbrace{\lb\frac{2\pi}{p^+}\rb^{p^+}\omega_N\pi^{-\frac{N+p^+}{2}}\lb\frac{p^+}{p^-}\rb^{-\frac{N+p^+}{2}}\Gamma\lb\frac{N+p^+}{2}\rb}_{c_1}
\end{align}
and 
\begin{align}\label{est_2}
            \rnnn|x|^k|\varphi_c|^{p(x)}\,{\rm {\rm dx}}&=\rnnn |x|^k a^{p(x)} e^{-\frac{\pi}{p^+}|x|^2p(x)}\,{\rm {\rm dx}}\nonumber\\
            &\le a^{p^-}\rnnn |x|^{p^+} e^{-\frac{\pi}{p^+}|x|^2p^-}\,{\rm {\rm dx}}\no\\
            &\le a^{p^-}\underbrace{\lb\frac{2\pi}{p^+}\rb^{p^+}\omega_N\pi^{-\frac{N+p^+}{2}}\lb\frac{p^+}{p^-}\rb^{-\frac{N+p^+}{2}}\Gamma\lb\frac{N+p^+}{2}\rb}_{c_2}
\end{align}
Combining \eqref{est_1} and \eqref{est_2}, we get
    \begin{align}\label{est_3}
        \rnnn|\na\varphi_c(x)|^{p(x)}\,{\rm {\rm dx}}+\rnnn|x|^k|\varphi_c(x)|^{p(x)}\,{\rm {\rm dx}}\le a^{p^-}(c_1+c_2)
\end{align}
Let us consider the case $\lv u\rv_\X<1$, then using Lemma \ref{modular_ineq_space}, we get
\begin{align}\label{eq5}
        \lv u\rv^{p^+}_\X\le\rho_\X(u)=\rnnn|\na\varphi_c(x)|^{p(x)}\,{\rm {\rm dx}}+\rnnn|x|^k|\varphi_c(x)|^{p(x)}\,{\rm {\rm dx}}\le a^{p^-} (c_1+c_2)
\end{align}
Therefore, from \eqref{eq4} and \eqref{eq5}, we obtain $\lv u\rv_\X\le\sigma$ if 
\begin{equation*}
    c\le\frac{1}{p^+}\lb\frac{\sigma^{p^+}}{(c_1+c_2)}\rb^\frac{p^+}{p^-}.
\end{equation*}
Similarly, for the case $\lv u\rv_\X\ge1$, we get $\lv u\rv_\X\le\sigma$ if 
\begin{equation*}
    c\le\frac{1}{p^+}\lb\frac{\sigma^{p^-}}{(c_1+c_2)}\rb^\frac{p^+}{p^-}.
\end{equation*}
Consequently, choosing
\begin{equation*}
    c_1(\sigma)=\min\left\{ \frac{1}{p^+},\frac{1}{p^+}\lb\frac{\sigma^{p^+}}{(c_1+c_2)}\rb^\frac{p^+}{p^-},\frac{1}{p^+}\lb\frac{\sigma^{p^-}}{(c_1+c_2)}\rb^\frac{p^+}{p^-}\right\}
\end{equation*}
ensures that for all $0<c\le c_1(\sigma)$ we have $\varphi_c\in S(c)\cap B_\sigma,$ that is $S(c)\cap B_\sigma\neq\emptyset.$
\end{proof}

 We prove the following lemma, consequently excluding the possibility of the minimizers locating on the boundary of $S(c)\cap B_\sigma.$
\begin{lemma}\label{lemma_separation}
    Let $q(x)$ satisfies \eqref{condition_on_q1}. Then for any fixed $\sigma > 0$ if $S(c) \cap (B_{\sigma} \setminus B_{a_2\sigma}) \neq \emptyset$, then there exists a constant $c_0 > 0$ (which may depend on $\sigma$) such that for any $0 < c < c_0$
$$\inf_{S(c) \cap B_{a_1\sigma}} E(u)<\inf_{S(c) \cap (B_{\sigma} \setminus B_{a_2\sigma})} E(u),$$
			where $0 < a_1 < a_2 < 1$.
\end{lemma}
\begin{proof}
    For any $\sigma>0$, by Lemma \ref{lm_1}, if $0<c\le c_2(\sigma)<c_1(\sigma)$ then $S(c)\cap B_{a_1\sigma}\neq\emptyset,$ where
    \begin{equation*}
    c_2(\sigma)=\min\left\{ \frac{1}{p^+},\frac{1}{p^+}\lb\frac{(a_1\sigma)^{p^+}}{(c_1+c_2)}\rb^\frac{p^+}{p^-},\frac{1}{p^+}\lb\frac{(a_1\sigma)^{p^-}}{(c_1+c_2)}\rb\right\}.
\end{equation*}
Let $u\in S(c) \cap (B_{\sigma} \setminus B_{a_2\sigma})$, then 
\begin{align}\label{eq12}
    E(u)&=\int_{\mathbb{R}^N}\frac{|\nabla u|^{p(x)}}{p(x)} \,{\rm {\rm dx}}+\int_{\mathbb{R}^N}\frac{|x|^k|u|^{p(x)}}{p(x)}\,{\rm {\rm dx}}-\int_{\R^N}\frac{|u|^{q(x)}}{q(x)}\,{\rm {\rm dx}}\no\\
    &\ge \frac{1}{p^+}\int_{\mathbb{R}^N}|\nabla u|^{p(x)}\,{\rm {\rm dx}}+\frac{1}{p^+}\int_{\mathbb{R}^N}|x|^k|u|^{p(x)}\,{\rm {\rm dx}}-\frac{1}{q^-}\int_{\R^N}|u|^{q(x)}\,{\rm {\rm dx}}\no\\
    &=\frac{1}{p^+}\rho_\X(u)-\frac{1}{q^-}\rho(u).
\end{align}
We first consider the case $\lv u\rv_\X\ge1.$ From Lemma \ref{modular_ineq} and \ref{modular_ineq_space}, we obtain
\begin{align}\label{eq9}
    E(u)&\ge \frac{1}{p^+}\lv u\rv_\X^{p^-}-\frac{1}{q^-}\max\Big\{\lv u \rv_{L^{q(x)}(\rnn)}^{q^+},\lv u \rv_{L^{q(x)}(\rnn)}^{q^-}\Big\}.
\end{align}
On the other hand, since $u\in S(c),$ we have
  \begin{align*}
      cp^-\le\int_{\R^N}|u|^{p(x)}\le cp^+.
\end{align*}
That is 
\begin{equation}\label{eq6}
\lv u\rv_{L^{p(x)}(\rnn)}\le \max\Big\{(cp^+)^{(1/p^+)},(cp^+)^{(1/p^-)}\Big\}=(cp^+)^{(1/p^+)} \text{ since }cp^+<1.
\end{equation}
Using Lemma \ref{GN_inequality} and \eqref{embedding} in \eqref{eq6}, we get
\begin{align}\label{eq7}
   \|u\|_{L^{q(x)}(\R^N)}\le K_\al\|u\|^{1-\al}_{L^{p(x)}(\R^N)}\|\na u\|_{L^{p(x)}(\R^N)}^\al\le K_\al(cp^+)^{((1-\al)/p^+)}\|u\|_\X^\al.
\end{align}
We also note that 
\begin{align}\label{eq8}
\max\Big\{\lv u \rv_{L^{q(x)}(\rnn)}^{q^+},&\lv u \rv_{L^{q(x)}(\rnn)}^{q^-}\Big\}\no\\
&=
\max\Big\{{K_\al^{q^+}(cp^+)^{(q^+(1-\al)/p^+)}}\|u\|_\X^{\al q^+},{K_\al^{q^-}(cp^+)^{(q^-(1-\al)/p^+)}}\|u\|_\X^{\al q^-}\Big\}\no\\
&\le {K'_\al(cp^+)^{(q^-(1-\al)/p^+)}}\|u\|_\X^{\al q^+}.
\end{align}
Inserting \eqref{eq8} in \eqref{eq9}, we obtain
\begin{align}\label{eq10}
    E(u)&> \frac{1}{p^+}\lv u\rv_\X^{p^-}-\frac{1}{q^-}{K'_\al(cp^+)^{(q^-(1-\al)/p^+)}}\|u\|_\X^{\al q^+}\no\\
    &\ge\frac{1}{p^+}(a_2\sigma)^{p^-}-\frac{1}{q^-}{K'_\al(cp^+)^{(q^-(1-\al)/p^+)}}\sigma^{\al q^+}.
\end{align}
On the other hand when $u\in S(c) \cap (B_{a_1\sigma})$, then 
\begin{align}\label{eq11}
    E(u)&=\int_{\mathbb{R}^N}\frac{|\nabla u|^{p(x)}}{p(x)} \,{\rm {\rm dx}}+\int_{\mathbb{R}^N}\frac{|x|^k|u|^{p(x)}}{p(x)}\,{\rm {\rm dx}}-\int_{\R^N}\frac{|u|^{q(x)}}{q(x)}\,{\rm {\rm dx}}\no\\
    &\le \frac{1}{p^-}\int_{\mathbb{R}^N}|\nabla u|^{p(x)}\,{\rm {\rm dx}}+\frac{1}{p^-}\int_{\mathbb{R}^N}|x|^k|u|^{p(x)}\,{\rm {\rm dx}}-\frac{1}{q^+}\int_{\R^N}|u|^{q(x)}\,{\rm {\rm dx}}\no\\
    &\le \frac{1}{p^-}\rho_\X(u)\le \frac{1}{p^-}\lv u\rv_\X^{p^+}\no\\
    &\le \frac{1}{p^-}(a_1\sigma)^{p^+}.
\end{align}
Now, from \eqref{eq10} and \eqref{eq11}, we get
\begin{align*}
\frac{1}{p^-}(a_1\sigma)^{p^+}<\frac{1}{p^+}(a_2\sigma)^{p^-}-\frac{1}{q^-}{K'_\al(cp^+)^{(q^-(1-\al)/p^+)}}\sigma^{\al q^+}
\end{align*}
which simplifies to
\begin{align*}
c < \frac{1}{p^{+}}
\left(
    \frac{q^{-}}{K'_{\alpha}\sigma^{\alpha q^{+}}}
    \left(
        \frac{(a_{2}\sigma)^{p^{-}}}{p^{+}}
        -
        \frac{(a_{1}\sigma)^{p^{+}}}{p^{-}}
    \right)
\right)^{\frac{p^{+}}{q^{-}(1-\alpha)}}:=c_3(\sigma).
\end{align*}
We would like to poin out that in the above expression, we can alwys choose $0<a_1<a_2<1$ depending on $\sigma$ such that $\left(\frac{(a_{2}\sigma)^{p^{-}}}{p^{+}}-\frac{(a_{1}\sigma)^{p^{+}}}{p^{-}}
    \right)>0$. Indeed, choose $0<a_1<1$ small enough such that 
    $$a_1<\left(\frac{p^-}{p^+\sigma^{p^+-p^-}}\right)^{1/p^+}$$
    which implies
    $$\left(\frac{p^+}{p^-}\sigma^{p^+-p^-}a_1^{p^+}\right)^{1/p^-}<1.$$
Now, we choose 
$$\max\Bigg\{a_1,\left(\frac{p^+}{p^-}\sigma^{p^+-p^-}a_1^{p^+}\right)^{1/p^-}\Bigg\}<a_2<1.$$
Next, we can consider the case $\lv u\rv_\X<1.$ From \eqref{eq12}, we have
\begin{align}\label{eq13}
    E(u)&\ge \frac{1}{p^+}\lv u\rv_\X^{p^+}-\frac{1}{q^-}\max\Big\{\lv u \rv_{L^{q(x)}(\rnn)}^{q^+},\lv u \rv_{L^{q(x)}(\rnn)}^{q^-}\Big\}\no\\
    &\ge \frac{1}{p^+}\lv u\rv_\X^{p^+}-\frac{1}{q^-}\lv u \rv_{L^{q(x)}(\rnn)}^{q^-}.
\end{align}
Using \eqref{eq7} in \eqref{eq13}, we obtain
\begin{align}\label{eq14}
    E(u)&> \frac{1}{p^+}\lv u\rv_\X^{p^+}-\frac{1}{q^-}{K''_\al(cp^+)^{(q^-(1-\al)/p^+)}}\|u\|_\X^{\al q^-}\no\\
    &\ge\frac{1}{p^+}(a_2\sigma)^{p^+}-\frac{1}{q^-}{K''_\al(cp^+)^{(q^-(1-\al)/p^+)}}\sigma^{\al q^-}
\end{align}
On the other hand when $u\in S(c) \cap (B_{a_1\sigma})$, then 
\begin{align}\label{eq15}
    E(u)\le \frac{1}{p^-}(a_1\sigma)^{p^-}.
\end{align}
Now, from \eqref{eq14} and \eqref{eq15}, we obtain
\begin{align*}
\frac{1}{p^-}(a_1\sigma)^{p^-}<\frac{1}{p^+}(a_2\sigma)^{p^+}-\frac{1}{q^-}{K''_\al(cp^+)^{(q^-(1-\al)/p^+)}}\sigma^{\al q^-}
\end{align*}
which simplifies to
\begin{align*}
c < \frac{1}{p^{+}}
\left(
    \frac{q^{-}}{K''_{\alpha}\sigma^{\alpha q^{-}}}
    \left(
        \frac{(a_{2}\sigma)^{p^{+}}}{p^{+}}
        -
        \frac{(a_{1}\sigma)^{p^{-}}}{p^{-}}
    \right)
\right)^{\frac{p^{+}}{q^{-}(1-\alpha)}}:=c_4(\sigma).
\end{align*}
Using similar arguments as in the previous case, we can always choose $0<a_1<a_2<1$ depending on $\sigma$ such that $\left(
        \frac{(a_{2}\sigma)^{p^{+}}}{p^{+}}
        -
        \frac{(a_{1}\sigma)^{p^{-}}}{p^{-}}
    \right)>0.$

Hence, there exists $0<c_0<\min\{c_2,c_3,c_4\}$, such that for any $0<c<c_0$ the following holds
$$\inf_{S(c) \cap B_{a_1\sigma}} E(u)<\inf_{S(c) \cap (B_{\sigma} \setminus B_{a_2\sigma})} E(u).$$
This concludes the proof of the lemma.
\end{proof}
Now, we present the proof of the main existence result of this paper.
\begin{proof}[Proof of Theorem \ref{main_thm}]
Let $\{u_n\} \subset S(c) \cap B_\sigma$ be a minimizing sequence for $\gamma_c$. 
		By compact embedding, there exists $u_c \in \mathcal{X}$ such that 
		\begin{align*}
	&u_n \rightharpoonup u_c \text{ in } \mathcal{X}, \\
&u_n \to u_c \text{ in } L^{t(x)}(\mathbb{R}^N),\ t(x) \in [p(x), p^*(x)), \\
&u_n \to u_c \text{ a.e. in } \mathbb{R}^N.
		\end{align*}
This shows that $u_c \in S(c)$. Also, using \cite[Theorem 1.4]{fan_lp_wp_spaces}, we get
\begin{equation}\label{eq20}
    \lim_{n\to\infty}\int_{\R^N}\frac{|u_n|^{q(x)}}{q(x)}\,{\rm {\rm dx}}=\int_{\R^N}\frac{|u_c|^{q(x)}}{q(x)}\,{\rm {\rm dx}}.
\end{equation}
Next, by the weak lower semicontinuity of the $\mathcal{X}$-norm and Lemma \ref{lemma_separation}, it follows that
\[
\|u_c\|_{\mathcal{X}} \le \liminf_{n \to \infty} \|u_n\|_{\mathcal{X}} \le a_2 \sigma.
\]
Hence $u_c \in S(c) \cap B_\sigma$. Using the definition of $\gamma_c$, together with \eqref{eq20} and the weak lower semicontinuity of the modular $\tilde\rho_\X$, we obtain
\begin{align}\label{eq21}
    E(u_c)&=\int_{\mathbb{R}^N}\frac{|\nabla u_c|^{p(x)}}{p(x)} \,{\rm {\rm dx}}+\int_{\mathbb{R}^N}\frac{|x|^k|u_c|^{p(x)}}{p(x)}\,{\rm {\rm dx}}-\int_{\R^N}\frac{|u_c|^{q(x)}}{q(x)}\,{\rm {\rm dx}}\no\\
    &= \int_{\mathbb{R}^N}\frac{|\nabla u_c|^{p(x)}}{p(x)} \,{\rm {\rm dx}}+\int_{\mathbb{R}^N}\frac{|x|^k|u_c|^{p(x)}}{p(x)}\,{\rm {\rm dx}}-\liminf_{n\to\infty}\int_{\R^N}\frac{|u_n|^{q(x)}}{q(x)}\,{\rm {\rm dx}}\no\\
    &\le \liminf_{n\to\infty}\lb \int_{\mathbb{R}^N}\frac{|\nabla u_n|^{p(x)}}{p(x)}\,{\rm {\rm dx}} +\int_{\mathbb{R}^N}\frac{|x|^k|u_n|^{p(x)}}{p(x)}\,{\rm {\rm dx}}-\int_{\R^N}\frac{|u_n|^{q(x)}}{q(x)}\,{\rm {\rm dx}}\rb\no\\
    &= \liminf_{n \to \infty} E(u_n) = \gamma_c \le E(u_c).
\end{align}
This shows that $E(u_c) = \gamma_c$ and $u_n \to u_c$ in $\mathcal{X}$. Indeed, from \eqref{eq21}, we get
\begin{align}\label{eq22}
    \lim_{n \to \infty} E(u_n)&=\lim_{n\to\infty}\lb \int_{\mathbb{R}^N}\frac{|\nabla u_n|^{p(x)}}{p(x)} \,{\rm {\rm dx}}+\int_{\mathbb{R}^N}\frac{|x|^k|u_n|^{p(x)}}{p(x)}\,{\rm {\rm dx}}-\int_{\R^N}\frac{|u_n|^{q(x)}}{q(x)}\,{\rm {\rm dx}}\rb\no\\
    &=\int_{\mathbb{R}^N}\frac{|\nabla u_c|^{p(x)}}{p(x)} \,{\rm {\rm dx}}+\int_{\mathbb{R}^N}\frac{|x|^k|u_c|^{p(x)}}{p(x)}\,{\rm {\rm dx}}-\int_{\R^N}\frac{|u_c|^{q(x)}}{q(x)}\,{\rm {\rm dx}}.
\end{align}
Therefore, using \eqref{eq20} in \eqref{eq22}, we have
\begin{align*}
\lim_{n\to\infty}\lb \int_{\mathbb{R}^N}\frac{|\nabla u_n|^{p(x)}}{p(x)} +\int_{\mathbb{R}^N}\frac{|x|^k|u_n|^{p(x)}}{p(x)}\rb=\int_{\mathbb{R}^N}\frac{|\nabla u_c|^{p(x)}}{p(x)} +\int_{\mathbb{R}^N}\frac{|x|^k|u_c|^{p(x)}}{p(x)}.
\end{align*}
Thus, from Lemma \ref{lem:modular_cgs_space}, we conclude that $u_n \to u_c$ in $\mathcal{X}$.

By Lemma \ref{lemma_separation}, we know that $u_c \notin S(c) \cap \partial B_\sigma$ as 
$u_c \in B_\sigma$, where
\[
\partial B_\sigma:= \{u \in \mathcal{H} \mid \|u\|_{\mathcal{X}}= \sigma\}.
\]

\medskip
Then $u_c$ is a critical point of $E|_{S(c)}$. Hence, there exists a Lagrange 
multiplier $\lambda_c \in \mathbb{R}$ such that $(u_c, \lambda_c)$ is a pair of 
solutions to problem \eqref{problem} for any $0 < c < c_0$. 

Let $u_c\in\X$ be the weak solution of \eqref{problem}, we have
\begin{align}\label{eq16}
    E(u_c)&=\int_{\mathbb{R}^N}\frac{|\nabla u_c|^{p(x)}}{p(x)} \,{\rm {\rm dx}}+\int_{\mathbb{R}^N}\frac{|x|^k|u_c|^{p(x)}}{p(x)}\,{\rm {\rm dx}}-\int_{\R^N}\frac{|u_c|^{q(x)}}{q(x)}\,{\rm {\rm dx}}\no\\
    &\ge \frac{1}{p^+}\int_{\mathbb{R}^N}|\nabla u_c|^{p(x)}\,{\rm {\rm dx}}+\frac{1}{p^+}\int_{\mathbb{R}^N}|x|^k|u_c|^{p(x)}\,{\rm {\rm dx}}-\frac{1}{q^-}\int_{\R^N}|u_c|^{q(x)}\,{\rm {\rm dx}}\no\\
    &\ge \frac{1}{p^+}\rho_\X(u_c)-\frac{1}{q^-}\int_{\R^N}|u_c|^{q(x)}\,{\rm {\rm dx}}.
\end{align}
From Theorem \ref{Pohozaev_variable} and Remark \ref{error}, we get
\begin{align*}
    \left(\frac{N-p^+}{p^+}\right)\int_{\mathbb{R}^N}|\nabla u_c|^{p(x)} \,{\rm {\rm dx}}+&\lb\frac{N+k}{p^+}\rb\int_{\mathbb{R}^N}|x|^k|u_c|^{p(x)}\,{\rm {\rm dx}}\\
    &-\frac{N}{p^-}\int_{\R^N}\la |u_c|^{p(x)}\,{\rm {\rm dx}}-R\le \frac{N}{q^-}\int_{\R^N}|u_c|^{q(x)}\,{\rm {\rm dx}}.
\end{align*}
 Since $u_{c}\in \mathcal{X}$ is a weak solution of \eqref{problem}, we may choose
$u_{c}$ as a test function in the weak formulation. Substituting this into the 
previous expression yields
\begin{align*}
    \left(\frac{N-p^+}{p^+}\right)&\int_{\mathbb{R}^N}|\nabla u_c|^{p(x)}\,{\rm {\rm dx}} +\lb\frac{N+k}{p^+}\rb\int_{\mathbb{R}^N}|x|^k|u_c|^{p(x)}\,{\rm {\rm dx}}\\
    -\frac{N}{p^-}&\left[\int_{\rnn}|\grad u_c|^{p(x)}\,{\rm {\rm dx}}+\int_{\rnn}|x|^k|u_c|^{p(x)}\,{\rm {\rm dx}}-\int_{\rnn}|u_c|^{q(x)}\,{\rm {\rm dx}}\right]-R\le \frac{N}{q^-}\int_{\R^N}|u_c|^{q(x)}\,{\rm {\rm dx}}.
\end{align*}

After simplification, we obtain
\begin{align*}
    \left(\frac{N-p^+}{p^+}-\frac{N}{p^-}\right)\int_{\mathbb{R}^N}|\nabla u_c|^{p(x)}\,{\rm {\rm dx}}+\lb\frac{N+k}{p^+}-\frac{N}{p^-}\rb&\int_{\mathbb{R}^N}|x|^k|u_c|^{p(x)}\,{\rm {\rm dx}}-R\\
    &\le \lb\frac{N}{q^-}-\frac{N}{p^-}\rb\int_{\R^N}|u_c|^{q(x)}\,{\rm {\rm dx}}
\end{align*}
which implies
\begin{align}\label{eq17}
    -&\lb\frac{N(p^+-p^-)+p^+p^-}{Np^+(q^--p^-)}\rb\int_{\mathbb{R}^N}|\nabla u_c|^{p(x)}\,{\rm {\rm dx}}-\lb\frac{p^+-p^-}{p^+(q^--p^-)}\rb\int_{\mathbb{R}^N}|x|^k|u_c|^{p(x)}\,{\rm {\rm dx}}\no\\
    +&\lb\frac{p^-}{N(q^--p^-)}\rb\lb\frac{k}{p^+}\rb\int_{\mathbb{R}^N}|x|^k|u_c|^{p(x)}\,{\rm {\rm dx}}
    -\lb\frac{Rp^-}{N(q^--p^-)}\rb\le-\frac{1}{q^-}\int_{\R^N}|u_c|^{q(x)}{\rm {\rm dx}}.
\end{align}
Substituting from \eqref{eq17} to \eqref{eq16} yields
\begin{align}\label{eq18}
    E(u_c)&\ge \frac{1}{p^+}\rho_\X(u_c)-\lb\frac{N(p^+-p^-)+p^+p^-}{Np^+(q^--p^-)}\rb\int_{\mathbb{R}^N}|\nabla u_c|^{p(x)}\,{\rm {\rm dx}}-\lb\frac{p^+-p^-}{p^+(q^--p^-)}\rb\int_{\mathbb{R}^N}|x|^k|u_c|^{p(x)}\,{\rm {\rm dx}}\no\\
    &\qquad+\lb\frac{p^-}{N(q^--p^-)}\rb\lb\frac{k}{p^+}\rb\int_{\mathbb{R}^N}|x|^k|u_c|^{p(x)}\,{\rm {\rm dx}}
    -R\lb\frac{p^-}{N(q^--p^-)}\rb
\end{align}
Using Remark~\ref{error}, and after dropping the nonnegative 
term in \eqref{eq18}, we obtain from the definition of the modular function 
$\rho_{\mathcal{X}}$ that
\begin{align}\label{eq19}
    E(u_c)&>\left[\frac{1}{p^+}-\lb\frac{N(p^+-p^-)+p^+p^-}{Np^+(q^--p^-)}\rb-\lb\frac{p^+-p^-}{p^+(q^--p^-)}\rb\right]\rho_\X(u_c).
\end{align}
Therefore, $E(u_c)>0,$ given that
\begin{equation*}
    q^- > 2p^+ - p^- + \frac{p^+p^-}{N}.
\end{equation*}
Next, we aim to prove that $u_c$ is the ground state solution. We prove by contradiction. Let $v_c\in S_c$ such that 
\begin{align*}
(E|_{S_c})'(v_c)=0 \quad\text{ and }E(v_c)<\gamma_c.
\end{align*}
That is, $v_c$ satisfies \eqref{problem} in a weak sense for some $\la\in \R.$ Following a similar arguments as above from \eqref{eq19}, we get
\begin{align*}
    E(v_c)\ge\left[\frac{1}{p^+}-\lb\frac{N(p^+-p^-)+p^+p^-}{Np^+(q^--p^-)}\rb-\lb\frac{p^+-p^-}{p^+(q^--p^-)}\rb\right]\rho_\X(v_c).
\end{align*}
Thus from \eqref{est_3}, we get
\begin{align*}
    \left[\frac{1}{p^+}-\lb\frac{N(p^+-p^-)+p^+p^-}{Np^+(q^--p^-)}\rb-\lb\frac{p^+-p^-}{p^+(q^--p^-)}\rb\right]\rho_\X(v_c)\le E(v_c)&<\ga_c\le E(\varphi_c)\\
    &\le (cp^+)^{p^-}(c_1+c_2)\\
    &\to0 \text{ as } c\to 0.
\end{align*}
Therefore, there exist $0<c_*<c_0$ such that $v_c\in B_\sigma$ for all $0<c<c_*$ and $E(v_c)\ge \ga_c.$ This contradicts the assumption $E(v_c)< \ga_c$. Hence, $u_c$ is a ground state of problem \eqref{problem} for $\la=\la_c\in \R.$ Moreover, using a similar argument as above, together with \eqref{eq19}, we obtain
\begin{align*}
    \left[\frac{1}{p^+}-\lb\frac{N(p^+-p^-)+p^+p^-}{Np^+(q^--p^-)}\rb-\lb\frac{p^+-p^-}{p^+(q^--p^-)}\rb\right]\rho_\X(u_c)\le E(u_c)&<\ga_c\le E(\varphi_c)\\
    &\le (cp^+)^{p^-}(c_1+c_2)\\
    &\to0 \text{ as } c\to 0.
\end{align*}
This implies that $\rho_\X(u_c) \to 0$ as $c\to 0.$ Using Lemma \ref{lem:modular_cgs_space}, we conclude that $\lv u\rv_\X\to0$ as $c\to 0.$ This completes the proof of the main result.
\end{proof}

\section{Appendix}
In this section, we provide a proof of the result \eqref{cgt_result} stated in Remark~\ref{error}.
\begin{lemma}\label{lm_est_proof}
    Ler $R$ as defined as in Remark \ref{error}, and assume that $p(x),q(x)\in \mathcal{P}$, then $$R \to 0 \quad \text{as } r_0 \to 0.$$
\end{lemma}
\begin{proof}
Let $p(x)\in \mathcal{P}$ then for $x\neq 0$, we get
\[
x \cdot \nabla p(x) = |x| \,\eta'(|x|)\,(p(x)-p_0), \qquad r_0<r<2r_0,
\]
since, $p(x)$ is Lipschitz continuous, we have
\[
|x\cdot \nabla p(x)| \le C_0, \qquad r_0<r<2r_0.
\]
 For convenience, we use the following notation 
	\begin{align*}
		R_1 &= \int_{\mathbb{R}^N}\left(\ln|u| - \frac{1}{p(x)}\right)|u|^{p(x)}\frac{(x \cdot \nabla p(x))}{p(x)}\,{\rm {\rm dx}}\\
        R_2&=\int_{\mathbb{R}^N}\left(\ln|u| - \frac{1}{q(x)}\right)|u|^{q(x)}\frac{(x \cdot \nabla q(x))}{q(x)}\,{\rm {\rm dx}}\\
        R_3&=\int_{\mathbb{R}^N}\left(\ln|u| - \frac{1}{p(x)}\right)|u|^{p(x)}|x|^k\frac{(x \cdot \nabla p(x))}{p(x)}\,{\rm {\rm dx}}\\
        R_4&=\int_{\mathbb{R}^N}\left(\ln|\nabla u| - \frac{1}{p(x)}\right)|\nabla u|^{p(x)}\frac{(x \cdot \nabla p(x))}{p(x)}\,{\rm {\rm dx}}
\end{align*}
\textbf{Estimates for $R_1$ and $R_2$.}
\begin{align*}
		R_1= \int_{\mathbb{R}^N}\left(\ln|u| - \frac{1}{p(x)}\right)|u|^{p(x)}\frac{(x \cdot \nabla p(x))}{p(x)}\,{\rm {\rm dx}}.
\end{align*}
Using $|x\cdot \nabla p(x)| \le C_0$ and $p(x)\ge p_->1$, we have
\[
|R_1| \le C \int_{A(r_0,2r_0)} |u|^{p(x)}\,(1+|\ln|u||)\,{\rm {\rm dx}},
\]
where $A(r_0,2r_0)=\{x\in\rnn: r_0<|x|<2r_0\}$. Therefore, $R_1\to0 \text{ as } r_0\to0.$
Since, $|A(r_0,2r_0)|\to0$ as $r_0\to0$ and $u\in \X$ this implies that $|u|^{p(x)}\,(1+|\ln|u||)\in L^1(A(r_0,2r_0))$. Using the similar argument we get $R_2\to0 \text{ as } r_0\to0.$

\textbf{Estimate for $R_3$.}
\[
R_3 = \int_{\rnn} \left(\frac{\ln|u|}{p(x)}-\frac{1}{p(x)^2}\right)(x\cdot\nabla p)|x|^k|u|^{p(x)}\,{\rm {\rm dx}}.
\]
Again, using $|x\cdot \nabla p(x)| \le C_0$ and $p(x)\ge p_->1$ gives
\[
|R_3| \le C(2r_0)^k \int_{A(r_0,2r_0)} |u|^{p(x)}\,(1+|\ln|u||)\,{\rm {\rm dx}}.
\]
Now, using the arguments for $R_1$ and $R_2$, we get
$$R_3\to0 \text{ as } r_0\to0.$$
\textbf{Estimate for $R_4$.}
\[
R_4 = \int_{\Omega}\left(\frac{\ln|\nabla u|}{p(x)}-\frac{1}{p(x)^2}\right)(x\cdot\nabla p)|\nabla u|^{p(x)}\,{\rm {\rm dx}},
\]
so that
\[
|R_4| \le C \int_{A(r_0,2r_0)} |\nabla u|^{p(x)}(1+|\ln|\nabla u||)\,{\rm {\rm dx}}.
\]
 Now, using the regularity results $u\in C^{1,\alpha}_{loc}(\rnn)$, we have $|\nabla u|^{p(x)}(1+|\ln|\nabla u||)\in L^1(A(r_0,2r_0)).$ Therefore, $R_4\to0 \text{ as } r_0\to0.$

 Adding up the above estimates, we conclude that $R \to 0 \quad \text{as } r_0 \to 0.$
 \end{proof}
 \section*{Acknowledgement}
The author, Nidhi (PMRF ID - 1402685), is supported by the Ministry of Education, Government of India, under the Prime Minister’s Research Fellows (PMRF) scheme. 
\bibliographystyle{abbrv}   
\bibliography{references}
\end{document}